\documentclass{article}
\usepackage[T1]{fontenc}
\usepackage[english]{babel}
\usepackage[utf8]{inputenc}
\usepackage{url}
\usepackage[colorlinks=true,linkcolor=blue,citecolor=magenta,urlcolor=blue]{hyperref}
\usepackage{stmaryrd}
\usepackage{amsthm}
\usepackage{amssymb}
\usepackage{MnSymbol}
\usepackage{amsfonts}
\usepackage{amsmath}
\usepackage{mathrsfs}
\usepackage{tikz}
\usepackage{lineno}
\usepackage{thmtools}
\usepackage{tikz-cd}
\usepackage[all,cmtip]{xy}
\usepackage{nameref, cleveref}

\newtheorem{Theorem}{Theorem}[section]

\newtheorem{Corollary}[Theorem]{Corollary}
\newtheorem{Lemma}[Theorem]{Lemma}
\newtheorem{Proposition}[Theorem]{Proposition}

\theoremstyle{definition}
\newtheorem{Example}[Theorem]{Example}

\theoremstyle{definition}
\newtheorem{Definition}[Theorem]{Definition}

\theoremstyle{remark}
\newtheorem{Remark}[Theorem]{Remark}



\def \C{\mathbb{C}}
\def \Z{\mathbb{Z}}
\def \R{\mathbb{R}}

\def \Q{\mathbb{Q}}
\def \A{\mathbb{A}}

\def \E{\mathbb{E}}

\def \H{\mathcal{H}}
\def \G{\mathbb{G}}
\def \g{\mathfrak{g}}

\def \L{\mathcal{L}}
\def \P{\mathbb{P}}

\def \O{\mathcal{O}}
\def \C{\mathbb{C}}
\def \Z{\mathbb{Z}}
\def \R{\mathbb{R}}

\def \Q{\mathbb{Q}}
\def \A{\mathbb{A}}
\def \a{\mathfrak{a}}
\def \AA{\mathcal{A}}

\def \E{\mathbb{E}}

\def \H{\mathcal{H}}
\def \G{\mathbb{G}}

\def \n{\mathfrak{n}}

\def \L{\mathcal{L}}
\def \P{\mathbb{P}}
\def \PP{\mathcal{P}}

\def \O{\mathcal{O}}
\def \X{X^{\text{BB}}}

\def \W{\textbf{W}}
\def \ti{\tilde}

\def \Hom{\mathrm{Hom}}

\def \M{\mathcal{M}}
\def \CC{\mathcal{C}}

\def \lim{\mathrm{lim}}
\def \colim{\mathrm{colim}}

\def \D{\mathcal{D}}

\def \BSX{\widetilde{X}}
\def \hX{\widehat{X}}

\def \d{\underline{\dim}}
\def \Icl{\mathrm{Icl}}
\def \MHM{\mathrm{MHM}}
\def \wD{{}^wD^}
\def \aa{\underline{a}}

\def \im{\mathrm{im}}
\def \rat{\mathrm{rat}}

\usepackage{microtype} 
\begin{document}
\title{Weighted Cohomology, Hodge Theory and Intersection Cohomology of Shimura varieties}
\author{Mingyu Ni}
\date{}
\maketitle

\begin{abstract}
We prove that the intersection cohomology of the Baily--Borel compactification of a complex Shimura variety is identified with the top weight quotient of the mixed Hodge structure on the reductive Borel--Serre compactification. This yields canonical cup products and functorial pullbacks on the intersection cohomology.
As an application, we introduce canonical cycle classes associated to special cycles, relating analytic geometric volumes of non-compact Shimura varieties to topological terms.
\end{abstract}
\begin{center}
    \textit{In memory of Günter Harder (1938–2025).}
\end{center}
\tableofcontents

%\linenumbers
\section{Introduction}

Owing to its desirable properties summarized in the Kähler package (\cite{cheeger1982l2}), intersection cohomology (of middle perversity) has become a basic tool in the study of singular spaces in complex geometry. However, despite their geometric origin, intersection cohomology groups do not admit canonical pullbacks or cup products within the abstract formalism of perverse sheaves. This lack of operations prevents intersection cohomology from behaving like the cohomology of smooth projective varieties, in particular, from allowing the definition of canonical cycle classes.

The aim of this paper is to restore these basic operations with the expected properties for the intersection cohomology of Shimura varieties, by means of Hodge theory. When $X$ is a complex Shimura variety (Hermitian locally symmetric variety), the intersection cohomology of its Baily--Borel compactification $\X$ is a natural tool in the study of the Langlands correspondence over number fields. Zucker's conjecture, proved independently by Looijenga (\cite{looijenga19882}) and Saper--Stern (\cite{saper19902}), asserts that it agrees with the $L^2$-cohomology
\begin{equation*}
    IH^k(\X;\C)\cong H^k_{(2)}(X;\C),
\end{equation*}  
where the latter admits a spectral decomposition of the discrete spectrum of automorphic forms. 

Pullbacks have already been considered by Bergeron in \cite{bergeron2006proprietes} via $L^2$-cohomology. The construction is based on an estimate of $L^2$-harmonic forms, however, it cannot extend to the cochain complexes of all $L^2$-forms. On the other hand, it is not clear if this construction preserves $\Q$-structures and satisfies desired properties such as the projection formula.

Our approach proceeds via the reductive Borel--Serre compactification $\hX$. Although $\hX$ is not an algebraic variety in general, it is well-known that its cohomology carries a functorial mixed Hodge structure, and is even the Betti realization of a motive defined over the reflex field, as developed by Ayoub--Zucker in \cite{ayoub2012relative} and further refined by Nair--Vaish in \cite{nair2015weightless}. The main result of this article is that the intersection cohomology is canonically identified with the top weight quotient of the mixed Hodge structure on $\hX$ (see \Cref{Comparison to RBS} for the statement with general coefficients)
\begin{Theorem}\label{Theorem 1.1}
\begin{equation*}
    Gr^W_kH^k(\hX;\Q)\cong IH^k(\X;\Q).
\end{equation*}
\end{Theorem}
Thus the mixed Hodge structure on $\hX$ provides a canonical source of functorial operations on $IH^*(\X;\Q)$. A further consequence of this weight comparison is that, for any resolution of singularities $X'$ of $\X$, e.g., a toroidal compactification, there is a \textit{canonical} injection $IH^*(\X;\Q)\hookrightarrow H^*(X';\Q)$. Such an injection is not canonical via the decomposition theorem. This allows us to introduce the notion of modular intersection class (\Cref{modular intersection class}) associated to special cycles, i.e., linear combination of Shimura subvarieties and their Hecke translations. It is then a canonical class in the intersection cohomology, and is universal in all resolutions of $\X$. In \cite{barthel1995relevement}, they show that the cycle classes can be lifted to intersection cohomology, but it is not canonical.

\textbf{Applications.} We present two arithmetic applications of these constructions on the intersection cohomology. The first concerns the automorphic Lefschetz properties, which could be understood as a branching problem. The main results in \cite{venkataramana2001cohomology} in the Hermitian setting can be extended verbatim to the non-compact Shimura varieties provided that the functoriality of the intersection cohomology has been established. In particular, this immediately resolves the open question posed in \cite{nair2023automorphic}, on the Lefschetz property of unitary Shimura varieties, in the missing degrees and also for different $\Q$-ranks (\Cref{Auto Lef Remark}). In fact, in these degrees, a new phenomenon arises: we show that for Shimura varieties $Y\subseteq X$ associated to unitary groups $U(m,1)\subseteq U(n,1)$, the image of the pullback $H^m(X;\C)\to H^m(Y;\C)$ lands in the inner cohomology $H^m_!(Y;\C):=\im (H^m_c(Y;\C)\to H^m(Y;\C))$, and it becomes an injection under Hecke translations. This seems to be related to the theory of Eisenstein cohomology and values of $L$-functions.

As a further application, we give an intersection-cohomological interpretation of the Siegel--Weil formula. In \cite{kudla49special}, Kudla obtains an analytic identity relating volumes of special cycles to the special values of Eisenstein series. Directly following the formula, he remarks that \textit{"It remains to give a cohomological interpretation of the left side of this identity\footnote{i.e., the generating series of the volumes of cycles.} in the noncompact case."} We show that Kudla's Siegel--Weil formula can be interpreted via the generating series of modular intersection classes (\Cref{Cohomological interpretation of Siegel--Weil}). This is based on the observation that the analytic geometric volumes of modular cycles can be interpreted as the degrees of the intersection cohomology classes. 

\textbf{On the Proof of \Cref{Theorem 1.1}.} The proof relies on a synthesis of established techniques rather than new methods. On the other hand, while these technical ingredients are standard, the overarching conclusion of \Cref{Theorem 1.1} does not seem to be implicitly known in the literature. 

It requires the general framework of weight truncations on mixed Hodge modules to construct the map and show the injectivity. The surjectivity is proved analytically by base change to $\C$. The bridge from Hodge theory to analysis is via  weighted cohomology introduced by Goresky--Harder--MacPherson in \cite{goresky1994weighted}. 

There is a canonical morphism of mixed Hodge structures $H^k(\hX;\Q)\to IH^k(\X;\Q)$, by realizing both as certain weighted cohomology. The weighted complexes are sheaves on $\hX$ truncated by weights of Lie algebras, and their direct images on $\X$ underlie canonical mixed Hodge modules. The relation is via a correspondence between weights of Lie algebras and Hodge weights. The image of that morphism then identifies with $Gr^W_kH^k(\hX;\Q)$; hence the injectivity of the map
\begin{equation*}
    Gr^W_kH^k(\hX;\Q)\hookrightarrow IH^k(\X;\Q),
\end{equation*}
 see \Cref{Theorem: Lie-Hodge correspondence}. These are all due to Nair (\cite{nair2012mixed}). Since the proof in \cite{nair2012mixed} has several flaws, see \Cref{Flaw of Nair's proof}, we present a complete proof in this article.

It remains to prove, after base change to $\C$, that the morphism $H^k(\hX;\C)\to IH^k(\X;\C)$ induced by the morphism of weighted complexes, is surjective. Over $\C$, weighted complexes can be constructed by certain differential forms on $X$ (called special forms in \cite{goresky1994weighted}). The first step is to establish an embedding of sheaves from $\W^0C(\C)$ into the sheaf of $L^p$-forms, for any $p\geq 1$. Here $\W^0C(\C)$ is the weighted complex computing the cohomology of $\hX$. The second step is to show that the injection is an eventual quasi-isomorphism, i.e., it becomes a quasi-isomorphism for sufficiently large $p$. The third step is to apply Zucker's conjecture, by realizing $IH^k(\X;\C)$ as the space of $L^2$-harmonic forms. The estimate in \cite{bergeron2006proprietes} can be applied to show that every $L^2$-harmonic form is $L^p$, and this concludes the proof. This proof applies to arbitrary locally symmetric spaces, not necessarily Hermitian, and shows that there is a canonical surjection from the cohomology of $\hX$ to the reduced $L^2$-cohomology; see \Cref{p to 2 surjectivity}.

Altogether, these steps present a synthesis of established methods. For the first step, the integrability (particularly when $p=2$) of certain differential forms is established by Borel (\cite{borel1974stable}). Following this idea, the first step is almost elementary. The second step adapts Zucker's argument for trivial coefficients (\cite{zucker2001reductive}). 

\textbf{Final Remarks and Questions.}

$\bullet$ Our current proof of \Cref{Theorem 1.1} is analytic and global (using $L^2$-harmonic forms). A possible
 algebraic and local proof seems to be very interesting for the behavior of the intersection complex, see \Cref{algebraic proof?}.

$\bullet$ We do not know if the generating series of the modular intersection classes are holomorphic modular forms in higher degrees, and the results in \cite{greer2025cohomological} indicate the same problem in the unitary case via the toroidal compactification, up to the same bound of degrees. We expect the generating series to be at least quasi-modular forms in higher degrees, see \Cref{Kudla remark}. In any case, it will be interesting to explore the relation between the strategy here via the intersection cohomology and the (corrections of) cycles developed in \cite{greer2025cohomological}.

\textbf{Organization of the paper.} 
In \Cref{section: Preliminaries}, we review the necessary background on mixed Hodge modules and weighted cohomology. In \Cref{section: Analytic Aspects of Weighted Cohomology}, we develop the analytic tools needed to establish the surjectivity. In \Cref{section: Algebraic Theory of Weighted Cohomology}, we study the algebraic theory of weighted cohomology, which implies the injectivity. In \Cref{section: Weight Comparison and Intersection Cohomology}, we shall state and prove the main result \Cref{Theorem 1.1} in the Hermitian setting, and then construct the induced operations on the intersection cohomology, and compare the degree map of the intersection cohomology with the analytic geometric volume. In \Cref{Section: Applications}, we apply these operations to the automorphic Lefschetz properties and the Siegel--Weil formula.

\textbf{Acknowledgments.} I am deeply indebted to my advisor Fabian Januszewski for introducing me the topic on the cohomology of arithmetic groups, and his guidance and encouragement throughout the development of the work. I am grateful to Jiajun Shi and Rebekka Strathausen for valuable discussions. The research was supported by the Deutsche Forschungsgemeinschaft (DFG, German Research Foundation) -- SFB-TRR 358/1 2023 -- 491392403.

\section{Preliminaries}\label{section: Preliminaries}
In this section, we will introduce the basic tools and fix notation. Specifically, the first two subsections are a summary of mixed Hodge modules and the weight truncation functors on them. These provide the Hodge-theoretic input for stating our result. The last two subsections are a summary of the theory of weighted cohomology on locally symmetric spaces, and it will provide the analytic framework to prove the weight comparison. 

\subsection{Generalities on Hodge Theory}\label{subsection: Generalities on Hodge Theory}

In this subsection, we recall some facts about mixed Hodge modules, introduced by Saito in \cite{saito1988modules} and \cite{saito1990mixed}. Although logically polarizable Hodge modules must be introduced before mixed Hodge modules, we focus only on the aspects required for our applications, following \cite{saito1989introduction}.

For a complex algebraic variety $X$, denote by $D^b_c(\Q_X)$ the derived category of bounded $\Q$-complexes with constructible cohomology. Saito associates to $X$ an abelian category $\MHM(X)$, the category of (algebraic) mixed Hodge modules, together with a faithful and exact functor
\begin{equation*}
    rat:\MHM(X)\to Perv(\Q_X),
\end{equation*}
which extends to a functor $\rat: D^b\MHM(X)\to D^b_c(\Q_X)$ satisfying
\begin{equation*}
    \rat\circ H^i \cong {}^pH^i\circ \rat,
\end{equation*}
where ${}^pH^i$ is the functor of perverse cohomology.

When $X$ is smooth, an object $M=(\M,F,K,\alpha,W)$ in $\MHM(X)$ consists of 
\begin{itemize}
    \item [(i.)] $(\M,F)$ is an increasingly filtered regular holonomic $\D_X$-module;
    \item [(ii.)] $K=\rat(\M)\in Perv(\Q_X)$, together with an isomorphism $\alpha:DR(\M)\simeq K\otimes_\Q \C$. Here $DR$ denotes the de Rham functor in Riemann-Hilbert correspondence;
    \item [(iii.)] $W$ is a pair of finite increasing weight filtrations on $\M$ and $K$.
 \end{itemize}
These data satisfy several technical properties, which we will not list here. When $X$ is singular, one embeds it into a nonsingular variety, and considers only $\D$-modules set-theoretically supported on $X$, by Kashiwara's theorem. 

\begin{Example}
    The category $\MHM(pt)$ of mixed Hodge modules on a point is equivalent to the category of polarizable $\Q$-mixed Hodge structures introduced by Deligne in \cite{deligne1971theorie}. 
\end{Example}

There exists a functorial, finite, increasing weight filtration $W$ on $\MHM(X)$; moreover, the functors $W_i$ and $Gr_i^W$ are exact (\cite[1.5]{saito1989introduction}). An object $M\in D^b\MHM(X)$ is said to be mixed of weights $\leq n$ (resp. $\geq n$) if $Gr^W_iH^j(M)=0$ for $i>j+n$ (resp. $i<j+n$), and pure of weight $n$ if $Gr_i^WH^jM\neq 0$ only when $i=j+n$. 

All standard functors on the derived category of constructible sheaves are lifted to functors on the derived category of mixed Hodge modules. In particular, for a morphism $f:Y\to X$ of complex varieties, one has
\begin{equation*}
    f_*,f_!:D^b\MHM(Y)\to D^b\MHM(X),\ \ \ \ f^*,f^!:D^b\MHM(X)\to D^b\MHM(Y),
\end{equation*}
satisfying the usual adjunction relations and are compatible with functors on $D^b_c$ under $rat$. Moreover, these functors are compatible with the weight structure.

\begin{Proposition}\label{Functors compatible with weights}(\cite[1.7]{saito1989introduction})
    If $M$ is of weight $\leq n$ (resp. $\geq n$), then so are $f_!M, f^*M$ (resp. $f_*M,f^!M$).
\end{Proposition}

This recovers the classical mixed Hodge theory on the Betti cohomology of complex varieties. Let $\Q^H\in \MHM(pt)$ denote the object corresponding to the trivial Hodge structure. The mixed Hodge structures on (Borel-Moore) homology and (compactly supported) cohomology on a complex variety $X$ of are then obtained via the functors induced by the structure morphism  $a_X:X\to pt$. 

If $X$ is smooth, then the category of admissible variations of rational mixed Hodge structures (\cite[2.1]{saito1989introduction}) is a full subcategory of $\MHM(X)$ (\cite[2.2]{saito1989introduction}), consisting of \textit{smooth mixed Hodge modules}.
\begin{Definition}\label{smooth MHM}
    Let $X$ be a smooth variety, a mixed Hodge module $M\in \MHM(X)$ is \textit{smooth} if $\rat(M)[-\dim(X)]$ is a local system. $M\in D^b\MHM(X)$ is \textit{smooth} if $H^i(M)$ is smooth for all $i$.
\end{Definition}
The full subcategory of smooth mixed Hodge modules in $\MHM(X)$ is a Serre subcategory, since the same holds for shifted local systems in the category of perverse sheaves. In particular, subobjects and quotients of smooth mixed Hodge modules are smooth. In this article, all mixed Hodge modules on a smooth variety will be smooth. 

\begin{Lemma}\label{smooth stalk}
    Let $X$ be a smooth variety, and $M\in D^b\MHM(X)$ be smooth. Then for any closed point $i_x:x\hookrightarrow X$, 
    \begin{equation*}
        H^i(i_x^*M)\cong i_x^*H^{i+\dim(X)}(M)[-\dim(X)].
    \end{equation*}
\end{Lemma}

\begin{proof}
    This is a direct consequence of the faithfulness of $\rat$ on $\MHM(pt)$ and the isomorphism $\rat\circ H^i\cong {}^pH^i\circ \rat$.
\end{proof}

 Henceforth, we assume that $\E$ is a polarizable variation of rational Hodge structures of pure weight $w$ on $X$, and let $\E^H$ denote the corresponding smooth mixed Hodge module (pure of weight $w+\dim(X)$).

 We now turn to intersection cohomology. Let $X$ be an irreducible variety of dimension $d_X$, and $j:U\hookrightarrow X$ a smooth open subvariety. Let $\E$ be an admissible variation of rational Hodge structures of weight $w$ on $U$, which we view as a mixed Hodge module $\E^H$. Define the minimal extension
 \begin{equation*}
        IC_X(\E):=\mathrm{im}(j_!\E_U[d_X]\rightarrow j_*\E_U[d_X])[-d_X].
    \end{equation*}
This construction is independent of the choice of $U$ on which the variation of Hodge structures $\E$ is defined. Then $IC_X(\E)[d_X]\in Perv(X)$ underlies a Hodge module $IC_X(\E^H)[d_X]$ that is pure of weight $w+d_X$, obtained by replacing $\E$ with $\E^H$. When $X$ is proper, its hypercohomology $IH^k(X,\E)$ carries a natural polarizable pure Hodge structure of weight $k+w$ by \Cref{Functors compatible with weights}. 

    \begin{Remark}
        The definition of the intersection complex $IC_X(\E)$ here differs from the one in \cite[1.13]{saito1989introduction} by a shift of $d_X$. The definition there agrees with $IC_X(\E)[d_X]$ in our notation. The hypercohomology of the definition  here computes the intersection cohomology in the sense of  \cite[\S V]{borel1984intersection}.
    \end{Remark}

\subsection{Weight Truncations and Weightless Cohomology}\label{subsection: Weight Truncations and Weightless Cohomology}
We recall the weight truncation functors on varieties and stratified varieties, following \cite[\S 3]{morel2008complexes} and \cite[\S 3]{nair2015weightless}.

Let $X$ be a complex variety. For any $a\in \Z$, we define $\wD{\leq a}(X)$ (resp. $\wD{>a}(X)$) to be the full subcategory of $D^b\MHM(X)$ consisting of objects 
\begin{equation*}
    \wD{\leq a}(X):=\{M\in D^b\MHM(X)|H^i(M)\ \textrm{has weights} \leq a\ \textrm{for all}\ i.\}
\end{equation*}
(resp. 
\begin{equation*}
    \wD{> a}(X):=\{M\in D^b\MHM(X)|H^i(M)\ \textrm{has weights} > a\ \textrm{for all}\ i.\}
\end{equation*}

Then $(\wD{\leq a}, \wD{>a})$ defines a $t$-structure on $D^b\MHM(X)$, and we denote $w_{\leq a}$ and $w_{> a}$ to be the truncation functors. The following lemma is clear from the definition, see also \cite[Proposition 3.1.3]{morel2008complexes}.

\begin{Lemma}
    Let $M\in D^b\MHM(X)$. Then for any $i\in\Z$, applying $H^i$ to the weight truncation triangle gives a short exact 
    sequence
    \begin{equation*}
        0\to H^i(w_{\leq a}K)\to H^i(K)\to H^i(w_{>a}(K))\to 0 .
    \end{equation*}
    Moreover, $H^i(w_{\leq a}K)\cong w_{\leq a}H^i(K)$ and $H^i(w_{>a}K)\cong H^i(w_{>a}K)$.
\end{Lemma}

The following lemma then characterize the weight truncation for smooth modules.

\begin{Lemma}\label{weight truncation for smooth MHM}
    Let $X$ be a smooth variety, and $M\in D^b\MHM(X)$ be smooth. Then 
    \begin{itemize}
        \item [(i.)] $w_{\leq a}M, w_{>a}M$ are both smooth, hence the weight truncations define a t-structure on the full stable subcategory of smooth mixed Hodge modules;
        \item [(ii.)] $M\in {}^wD^{\leq a}(X)$ if and only if, for any closed point $i_x:x\hookrightarrow X$, $H^i(i_x^*M)$ has weights $a-\dim(X)$ for all $i$.
    \end{itemize}
\end{Lemma}

\begin{proof}
(i.) follows immediately by the previous lemma, and smooth mixed Hodge modules are closed under subobjects and quotients. For (ii.), by pointwise criterion \cite[(4.6.1)]{saito1990mixed}, $H^i(M)$ has weights $\leq a$ if and only if $i_x^*H^i(M)$ has weights $\leq a$ for any $x$. The result then follows by \Cref{smooth stalk}. 
\end{proof}

We warn that an object $M$ in $\wD{\leq a}(X)$ does not agree with $M$ has weights $\leq a$ in our previous definition, which requires that $H^i(M)$ has weights $\leq i+a$.

More generally, let $X=\coprod^r
_{j=0} X_j$ be a stratified complex variety such that $X_k$ is open in $X-\coprod_{j=0}^{k-1}X_j$. Let $\aa=(a_0,...,a_r)\in \Z^{r+1}$, one obtains a $t$-structure $(\wD{\leq \aa}(X),\wD{>\aa}(X))$ by gluing $t$-structures $(\wD{\leq a_j}(X_j),\wD{>a_j}(X_j))$ (\cite[1.4]{beilinson2018faisceaux}). More explicitly, an object $M$ in $\wD{\leq \aa}(X)$ (resp. in $\wD{>\aa}(X)$) if and only if $i_j^*M\in \wD{\leq a_j}(X_j)$ (resp. $i^!_jM\in \wD{>a_j}(X_j)$) for each $i_j:X_j\hookrightarrow X$.

In particular, there exist functors $w_{\leq \aa}$ and $w_{>\aa}$ on $D^b\MHM(X)$, such that for any object $M$, $w_{\leq \aa}M \in \wD{\leq\aa}(X)$ and $w_{>\aa}M\in \wD{>\aa}(X)$ with an exact triangle
\begin{equation*}
    w_{\leq \aa}M\rightarrow M\rightarrow w_{> \aa}M.
\end{equation*}
Clearly, $w_{\leq \aa}\circ w_{\leq \underline{b}}=w_{\leq \aa}$ for $\aa\leq \underline{b}$, that is, $a_j\leq b_j$ for all $j=0,...r$. We denote a constant tuple $(a,...a)$ simply by $a$. Let $\E$ be an admissible variation of Hodge structures of pure weight $w$ on some open dense smooth subvariety $j:U\hookrightarrow X$ with $j_*\E$ and $j_!\E$ are constructible. According to \cite[Theorem 3.1.4]{morel2008complexes}, we have
\begin{Proposition}\label{Morel Intersection} 
   \begin{equation*}
       IC_X(\E)=w_{\leq d_X+w}j_*\E.
   \end{equation*} 
\end{Proposition}

We shall also use the weight truncation for $\d:=(\dim(X_0),...\dim(X_r))$. When $\E=\Q$, we denote by
\begin{equation*}
    EC_X:=w_{\leq \underline{\dim}}j_*\Q_U=w_{\underline{\dim}}IC_X.
\end{equation*}
It is called weightless cohomology/motive in \cite{nair2015weightless} and \cite{vaish2016motivic}, and is called relative Artin motive in \cite{ayoub2012relative}. It is shown in \cite[Corollary 4.3.7]{vaish2016motivic} that they are equivalent in characteristic $0$.  This is an object which behaves like a slight modification of $\Q_X$, but unlike $IC_X$, it admits a ring structure, and enjoys much better functorial behavior. We summarize the basic properties that will be needed. 

\begin{Proposition}\label{Properties of EC}
    \begin{itemize}
        \item [(i.)] $EC_X$ is a ring object and there is a natural ring homomorphism $\Q_X\to EC_X$;

        \item [(ii.)] $EC_{X\times Y}\cong EC_{X}\boxtimes EC_Y$;

        \item [(iii.)] If $f:Y\to X$ is a morphism of varieties such that the closed image of each irreducible component of $Y$ meets the smooth locus of $X$, then there exists a natural pullback $f^{\sharp}:EC_X\to f_*EC_Y$ of ring objects, such that the diagram
        \begin{equation*}
\begin{tikzcd}
\Q_X \arrow[d] \arrow[r] & EC_X \arrow[d, "f^{\sharp}"'] \arrow[r] & j_{X*}\Q_{X_{reg}} \arrow[d] \\
f_*\Q_Y \arrow[r]        & f_*EC_Y \arrow[r]                       & f_*j_{Y*}\Q_{Y_{reg}}       
\end{tikzcd}
        \end{equation*}
        commutes. Here $j_X:X_{reg}\hookrightarrow X$ is the immersion from the smooth locus.
    \end{itemize} 
\end{Proposition}
\begin{proof}
    We refer to \cite[Theorem 3.3.10]{vaish2016motivic} for a summarization.
\end{proof}

For general $\E$ of weight $w$ on $U$, we have the following property of the weightless cohomology. This further justifies it as a modification of the ordinary cohomology of a projective variety. The proof is identical to \cite[Lemma 4.3.1]{nair2015weightless}, up to a uniform shift of weights by $w$.

\begin{Proposition}\label{weightless and intersection}
    When $X$ is proper, the hypercohomology $EH^k(X,\E):=\mathbb{H}^k(X,w_{\leq\d+w}j_*\E)$ has Hodge weights $\leq k+w$, and the canonical homomorphism $w_{\leq\d+w}j_*\E\to IC_X(\E)$
    induces an identification of the top weight quotient as the image in the intersection cohomology
    \begin{equation*}
        Gr^W_{k+w}EH^k(X,\E)\cong\im(EH^k(X,\E)\to IH^k(X,\E))\hookrightarrow IH^k(X,\E).
    \end{equation*}
\end{Proposition}

\subsection{Parabolic Subgroups and Weight Profiles}\label{subsection: Parabolic Subgroups and Weight Profiles}
This subsection is a summarization of some basic facts about parabolic subgroups of a reductive group, in order to fix notation. Let $G$ be a connected reductive algebraic group over $\Q$. For a rational parabolic subgroup $P\subseteq G$, let $S_P$ be the maximal $\Q$-split torus in the center of the Levi quotient $L_P=P/N_P$, where $N_P$ is the unipotent radical of $P$. If $P$ is conjugate to $P'$, then there is an isomorphism $S_P\cong S_{P'}$ that is independent of the choice of conjugation. Thus $S_P$ depends only on the conjugacy class of $P$. Moreover, if $P\subseteq Q$ then there is a natural embedding $S_Q\subseteq S_P$. Therefore, for any minimal parabolic subgroup $P_0$, $S_P$ is naturally identified as a subtorus of $S:=S_{P_0}$, \textit{the} maximal split torus. Henceforth we omit the subscript if $P=P_0$ unless the choice of $P_0$ is being emphasized.

For every (conjugacy class of) rational parabolic group $P$, There is a canonical finite set $\Delta_P\subseteq X(S_P):=\Hom(S_P,\G_m)$, \textit{the} set of simple roots in $\Phi(\n_P,S_P)$\footnote{It is not a root system when $P$ is not minimal.}. There is a one-to-one correspondence between the set of parabolic subgroups $Q$ containing $P$ and the set of subsets of $\Delta_P$: For each $Q$, one associates to a subset $\Delta^Q_P\subseteq \Delta_P$ consisting of roots restricting to $1$ on $S_Q$. It follows that 
\begin{equation*}
    S_Q=(\bigcap_{\chi\in \Delta^Q_P}\ker(\chi))^0.
\end{equation*}
Set $S^Q_P=(\bigcap_{\chi\in\Delta_Q\cup X(S_G)}\ker\chi\cap S_P)^0$, and then $S_P=S_QS^Q_P$. 

 Let $A_P:=S_P(\R)^0$ and $\a_P$ be its Lie algebra. Then $X(S_P)\otimes_\Z\R\cong\a_P^*$ and
 \begin{equation*}
     \a_P^*=\a_Q^*\oplus (\a^{Q}_P)^*.
 \end{equation*}
 In particular, for any parabolic subgroup $P$, there is an  isomorphism
\begin{equation*}
        (\a^{G}_P)^*\cong X(S^G_P)\otimes_\Z\R\xrightarrow{\sim} \R^{\Delta_P}.
\end{equation*}
  This defines a partial order on $(\a^{G}_P)^*$ induced by the product order on $\R^{\Delta_P}$. For any $P\leq Q$, and $\eta\in (\a^{G}_P)^*$, we define $\eta_Q\in (\a^{G}_Q)^*$ by restriction.

\begin{Definition}\label{weight profile}
    A \textit{weight profile} is an element $\eta\in (\a^{G})^*\cong \R^\Delta$. 
\end{Definition}

For any representation $E$ of $L_P$, and any weight profile $\eta$, we define the weight truncation functor as 
\begin{equation*}
    E_{\geq \eta}:=\bigoplus_{\chi\in X(S^G_P),\chi\geq\eta_P}E_{\chi},
\end{equation*}
by weight decomposition. Similarly, we could also define $E_{>\eta}$ and denote it by $E_{\geq \eta+\varepsilon}$. Here we could understand $\varepsilon>0$ as a sufficiently small weight profile, since all weights appearing in $E$ belongs to the lattice $X(S^G_P)$ in $(\a^G_P)^*$.

In this article, only $\eta=0$ and the half sum of all negative roots $\eta=-\rho$ will be considered. Moreover, we may always reduce to the standard assumption that $A_G$ is trivial (as in \cite[(3.5)]{goresky1994weighted}). This will save the notation, for example, we could write $\a_P^*$ in place of $(\a^G_P)^*$.

\subsection{Special Forms and Weighted Cohomology}\label{subsection: Special Forms and Weighted Cohomology}
We will follow \cite{goresky1994weighted} to discuss (reductive) Borel--Serre compactification, differential forms on them, and weighted cohomology. Let $G$ be a connected reductive algebraic group over $\Q$, and $G(\R)$ the real points of $G$, endowed with the real analytic Lie group topology. We will fix a maximal compact subgroup $K\subseteq G(\R)$.  Denote by $D$ the \textit{symmetric space} $G(\R)/KA_G$, and $\ti{D}$ the  Borel--Serre partial compactification. Let $\PP$ be the poset of parabolic subgroups of $G$, then $\ti{D}=\coprod_{P\in\PP}e_P$, where $e_P:=A_P\backslash D\cong D_{L_P}\times N_P(\R)$. 

Let $\Gamma\subseteq G(\Q)$ be an arithmetic group. We assume throughout that it is \textit{neat}, in the sense of \cite[17.1]{borel2019introduction}. Every arithmetic group contains a normal neat subgroup of finite index (\cite[17.4]{borel2019introduction}). Let $\Gamma\backslash\PP$ be the quotient poset, which is a finite set.  The \textit{locally symmetric space} $X:=\Gamma\backslash D=\Gamma\backslash G(\R)/KA_G$ is in general not compact, and the \textit{Borel--Serre compactification} $\BSX=\Gamma\backslash\ti{D}$ is a stratified space over $\Gamma\backslash \PP$ (i.e., $\BSX\to \Gamma\backslash\PP$ is continuous when the poset is equipped with the Alexandrov topology), such that for each $P\in\Gamma\backslash\PP$, the stratum is
\begin{equation*}
    Y_P:=\Gamma_P\backslash P(\R)/K_PA_P,
\end{equation*}
where $\Gamma_P=\Gamma\cap P(\R)$, and $K_P=K\cap P(\R)$. Let $\Gamma_{N_P}=\Gamma\cap N_P(\R)$, and then $\Gamma_{L_P}=\Gamma_P/\Gamma_{N_P}\subseteq L_P(\Q)$ is an arithmetic group. Let 
\begin{equation*}
    X_P:=\Gamma_{L_P}\backslash L_P(\R)/K_PA_P,
\end{equation*}
a locally symmetric space for $\Gamma_{L_P}$, which is a smooth manifold by our assumption that $\Gamma$ is neat. Then there is a flat fibration $Y_P\rightarrow X_P$ whose fibers are compact nilmanifolds $\Gamma_{N_P}\backslash N_P(\R)$. Let
\begin{equation*}
    \hX:=\coprod_{P\in \Gamma\backslash \PP}X_P, 
\end{equation*}
endowed with the quotient topology $p:\BSX\rightarrow \hX$. This is then a compact space stratified over $\Gamma\backslash\PP$, and is called the \textit{reductive Borel--Serre compactification} of $X$. Note that $Y_G=X_G=X$, and for any parabolic subgroup $P$, the closure of $X_P$ in $\hX$ is identified with the reductive Borel--Serre compactification $\hX_P$ of $X_P$. 

Let $P=Q_{1}\cap...\cap Q_{s}$ be a parabolic subgroup with $Q_i$ maximal parabolic. We recall that in \cite[(6.6)]{goresky1994weighted}, it is shown that there exists a "geodesic neighborhood" $\ti{U}_P$ of the closure $\ti{Y}_P$ in $\BSX$ with a diffeomorphism $\ti{U}_P\xrightarrow{\sim} \ti{Y}_P\times [0,2)^s$ such that $\ti{U}_P\to \ti{Y}_P$ is induced by the \textit{geodesic retraction} $r_P:\ti{e}_Q^*\to \ti{e}_P$ (\cite[(6.3.5)]{goresky1994weighted}). We shall denote by $U_P=\ti{U}_P\cap X$, and call it a \textit{geodesic neighborhood} of $Y_P$.\footnote{Literally, $U_P$ is not a neighborhood of $Y_P$, since they are disjoint in $\ti{X}$. On the other hand, all points in $X$ map to $Y_P$ via the geodesic retraction $r_P$.} In fact, geodesic neighborhoods are firstly defined for $\tilde{e}_P$ (\cite[(6.4)]{goresky1994weighted}), and we could similarly talk about geodesic neighborhoods of $e_P$ in $D$.

Let $E$ be a finite dimensional complex algebraic representation of $G$, then it associates to a local system $\E$ on $X$. Following \cite[(13.2)]{goresky1994weighted}, we say that a smooth differential $k$-form $\sigma$ on $X$ with values in $\E$ is \textit{special} if for each stratum $Y_P$ of $\BSX$, there exists a geodesic neighborhood $U_P$ of $Y_P$ such that the restriction of $\sigma$ on $U_P$ is the pullback of an $N_P$-invariant differential form $\sigma_P\in\Omega^k(Y_P,\E)^{N_P}$ on $Y_P$ via the geodesic retraction. We denote by $\Omega^\bullet_{sp}(\hX,\E)$ the sheaf of special differential forms on $\hX$, i.e., the sheafification of the underived pushforward of the presheaf of (restrictions of) special forms on $X$. By \cite[(13.6)]{goresky1994weighted}, it is quasi-isomorphic to $j_*\E$ for $j:X\hookrightarrow \hX$. In fact, the restriction of $\Omega^\bullet_{sp}(\hX,\E)$ to a boundary stratum $X_P$ consists of forms on $X_P$ with values in $C^\bullet(\n_P,E)$, considered as a local system on $X_P$ since the lie algebra cohomology (Koszul complex) $C^\bullet(\n_P,E)$ is naturally a representation of $L_P$; see \cite[\S 11, \S 12]{goresky1994weighted}. In other words,
\begin{equation*}
    \Omega^k_{sp}(\hX,\E)|_{X_P}=\bigoplus_{p+q=k}\Omega^{p}(X_P,C^q(\n_P,\E))=\bigoplus_{p+q=k}\Omega^p(Y_P,\E)^{N_P}.
\end{equation*}

It turns out that for any weight profile $\eta\in (\a^G)^*$, one can associate to a subsheaf $\Omega^\bullet_{sp}(\hX,\E)_{\geq\eta}$ of $\Omega^\bullet_{sp}(\hX,\E)$ by truncating the boundary stratum 
\begin{equation*}
    \Omega^\bullet_{sp}(\hX,\E)_{\geq\eta}|_{X_P}=\bigoplus_{p+q=k}\Omega^{p}(X_P,C^q(\n_P,\E)_{\geq\eta}).
\end{equation*}
This indeed defines a complex of fine sheaves $W^\eta C(\E)$ on $\hX$ (\cite[(14.1)]{goresky1994weighted}), and we call it the \textit{weighted complex of profile} $\eta$. Similarly, we can also define a subcomplex $\Omega^\bullet_{sp}(\hX,\E)_{>\eta}$, which we denote by $\W^{\eta+\varepsilon}C(\E)$. Clearly, when $\eta\leq \eta'$, there is an embedding of sheaves $\W^{\eta'} C(\E)\hookrightarrow \W^{\eta}C(\E)$. In particular, $\W^{\eta+\varepsilon}C(\E)\hookrightarrow W^\eta C(\E)$.

\begin{Example}
    When $\eta=0$ and $\E=\C$ the constant sheaf, the weighted complex $\W^0C(\C)$ is quasi-isomorphic to the constant sheaf $\C_{\hX}$. When $X$ is Hermitian and $\eta=-\rho$ is the middle weight profile, the pushforwards of the weighted complexes $\W^{-\rho+\varepsilon}C(\E)$ and $\W^{-\rho}C(\E)$  to the Baily--Borel compactification are isomorphic, and both are isomorphic to the intersection complex $IC_{\X}(\E)$; see \cite[(23.2)]{goresky1994weighted}. We also remark that $-\rho+\varepsilon$ (resp. $-\rho$) is denoted by the upper (resp. lower) middle weight profile $\mu$ (resp. $\nu$) in \cite{goresky1994weighted}. 
\end{Example}

Finally, the \textit{weighted cohomology} $\W^{\eta}H^k(X,\E)$ is defined as the hypercohomology $\mathbb{H}^k(\hX,W^\eta C(\E))$.

There is another definition of weighted complex in \cite{goresky1994weighted} for rational algebraic representation $E$, and they show that two approaches agree after base change to $\C$; see \cite[(29.2)]{goresky1994weighted}. 

\section{Analytic Aspects of Weighted Cohomology}\label{section: Analytic Aspects of Weighted Cohomology} 

\subsection{Borel's Integrability Criterion}\label{subsection: Borel's Integrability Criterion}
The aim of this subsection is to modify \cite[Proposition 5.5]{borel1974stable}. Here we will state it in the form intentionally for our latter purpose of the estimate on the weight profile $\eta=0$ and $-\rho$, and we do not pursue its generality. 

Recall that for any rational parabolic subgroup $P$ there is a diffeomorphism $\mu_P: A_P\times D_{L_P}\times N_P(\R)\xrightarrow{\sim} D$, with metric relation
\begin{equation*}
    \mu_P^*(dx^2)=da^2+dz^2+\sum_{\beta\in \Phi(\n_P,S_P)}a^{-2\beta}h_{\beta}(z),
\end{equation*}
where $\Phi(\n_P,A_P)$ is the set of all roots of $\n_P$ with respect to $A_P$, and $h_{\beta}$ is a left invariant smooth metric on the root space $\n_\beta:=\{X\in \n_P|Ad(a)X=a^{\beta}X\}$.\footnote{Here we follow the convention in \cite[(3.5)]{zucker19822} instead of \cite[(4.3)]{borel1974stable}, since we assume that the $G$-action is on the left, and the factor $2^{-1}$ in \cite{borel1974stable} is absorbed in the definition of $h_{\beta}$. This also requires to change the signs of certain weights below.}

Let $E$ be an algebraic complex representation of $G$. For a smooth form $\sigma$ on $D$ with values in $E$, the pullback $\mu_P^*\sigma$ defines a form on the product, which is then an element in
\begin{equation*}
    \Omega^\bullet(A_P\times D_{L_P})\hat{\otimes} C^\bullet(\n_P,\E),
\end{equation*}
since differential forms on $N_P(\R)$ are computed by the Koszul complex $\bigwedge^\bullet\n_P^*\otimes E= C^\bullet(\n_P,E)$. 

\begin{Definition}\label{weightless form}
    A smooth form $\sigma$ on $D$ is \textit{weightless} (resp. \textit{upper-middle}) with respect to $P$ if the weights of $\sigma$ in the Koszul complex $C^\bullet(\n_P,E)$ with respect to $A_P$ are all $\geq 0$ (resp. $>-\rho_P$).
\end{Definition}
In particular, for a special form $\sigma\in \W^0C(\E)$ (resp. $\sigma\in \W^{-\rho+\varepsilon}C(\E)$) on $X$, its pullback to $D$ is weightless (resp. upper-middle) with respect to all parabolic subgroups $P$.

This can also be described via coordinates, and we do so following \cite[(5.1)]{borel1974stable}. Let $m=\dim(D)$, and $\Delta_P=\{\alpha_1,...\alpha_s\}$. Fix a moving frame ($1$-forms) $(\omega^i)_{1\leq i\leq m}$ on $A_P\times D_{L_P}\times N_P(\R)$, where $\omega^i$ is lifted via projections, from $d\log\alpha_i$ on $A_P$ for $1\leq i\leq s$, from an orthogonal frame on $D_P$ (for $s< i\leq t$ for some index $t$), and from a set of left invariant $1$-forms on $N_P(\R)$ for $i>t$, which form a basis for each $\n_{\beta}^*\ (\beta\in \Phi(\n_P,S_P))$ at the origin $e\in N_P(\R)$.

For each $i\in \{1,...m\}$, we put $\alpha(i)=-\beta$ if $i>t$ and $\omega^i_e\in\n_{\beta}$, and $\alpha(i)=0$ if $i\leq t$. More generally, for a subset $J$ of $\{1,...m\}$, we set $\omega^J:=\wedge_{i\in J}\omega^i$ and $\alpha(J):=\sum_{i\in J}\alpha(i)$.

Fix a basis $\{e_j\}$ of $E$ consisting of weight vectors of $A_P$, with weights $\gamma_j$. Then for a smooth $k$-form $\sigma$ on $D$ with values in $E$, the pullback $\mu_P^*(\sigma)$ can be written as
\begin{equation*}
    \mu_P^*(\sigma)=\sum_{|J|=k,j}f_{J,j}e_j\omega^J, 
\end{equation*}
where $f_{J,j}$ is a smooth function on the product. It is then by definition that a form $\sigma$ on $D$ is weightless (resp. upper-middle) if and only if $\alpha(J)+\gamma_j\geq 0$ (resp. $\alpha(J)+\gamma_j>-\rho_P$) for each $f_{J,j}$.

For $t>0$, a \textit{Siegel set} $S_{P,t,\omega}$ in $D$ with respect to $P$ is of the form $\mu_P(A_{p,t}\times\omega)$, where $\omega$ is a relatively compact subset in $X_P\times N_P(\R)$, and 
\begin{equation*}
    A_{P,t}:=\{a\in A_P|a^\alpha\geq t\ \textrm{for all}\ \alpha\in \Delta_P\}.
\end{equation*}

The following is a trivial modification of \cite[Lemma 5.4]{borel1974stable}. Recall the notation that $f\prec g$ for two non-negatively real valued functions if there is a constant $c>0$ such that $f\leq cg$.

\begin{Lemma}\label{Lp-criterion of a function}
    Let $1\leq p<\infty$. Let $S_{P,t,\omega}$ be a Siegel set with respect to $P$, and $f$ a positive measurable function on $S_{P,t,\omega}$. Assume that $(f\circ\mu_P)(a,q)\prec a^{-\lambda}\ (a\in A_{P,t}, q\in \omega)$ for some $\lambda\in \a_P^*$, with $2\rho_P+p\lambda> 0$. Then $f$ is $L^p$ on $S_{P,t,\omega}$.
\end{Lemma}

Now we state the integrability criterion we shall use in what follows. This is a modification of \cite[Proposition 5.5]{borel1974stable}, and the proof will be identical.

\begin{Proposition}\label{Lp-cr of a form}
     Let $S_{P,t,\omega}$ be a Siegel set with respect to $P$, and $\sigma$ a smooth $k$-form on $S_{P,t,\omega}$ with values in $E$. Write as before
\begin{equation*}
    \tau:=\mu_P^*(\sigma)=\sum_{|J|=k,j}f_{J,j}e_j\omega^J.
\end{equation*}
\begin{itemize}
    \item [(i.)] Let $2\leq p<\infty$. Assume that $\sigma$ is weightless with respect to $P$, and $|f_{J,j}|\prec a^{-\lambda}$ for $\lambda\in \a_P^*$ with $2\rho_P+p\lambda> 0$. Then $\sigma$ is $L^p$ on $S_{P,t,\omega}$.

    \item [(ii.)] Assume that $\sigma$ is upper-middle with respect to $P$, and $|f_{J,j}|\prec a^{-\lambda}$ for $\lambda\in \a_P^*$ with $\lambda\geq 0$. Then $\sigma$ is $L^2$ on $S_{P,t,\omega}$.
\end{itemize}

\end{Proposition}
\begin{proof}
We denote the product $A_P\times D_{L_P}\times N_P(\R)$ by $Y$. We show that
\begin{equation*}
    \int_{A_{P,t}\times \omega}(\tau_y,\tau_y)^{p/2}dv_Y<\infty,
\end{equation*}
where $(\ ,\ )$ denotes the scalar product on $\wedge^qT^*Y$ associated to $dy^2$. Write
\begin{equation*}
    dy^2=\sum_{1\leq i,j\leq m}g_{ij}\omega^i\omega^j,
\end{equation*}
and $(g^{ij})$ the inverse matrix of $(g_{ij})$. For two strictly increasing sequences $J=\{j_1,...j_q\}, J'=\{j_1',...,j_q'\}$ in $\{1,...,m\}$, put
\begin{equation*}
    g^{J,J'}:=\det(g^{j_i,j_k'}),
\end{equation*}
and let $f^J:=\sum_{J'}g^{J,J'}f_{J'}$. Then
\begin{equation*}
    (\tau_y,\tau_y)^{p/2}=(\sum_{J,J'}g^{J,J'}f_Jf_{J'})^{p/2}=(\sum_{J}f^Jf_J)^{p/2},
\end{equation*}
where we write $f_J:=\sum_jf_{J,j}e_j$ as a function on $Y$ with values in $E$. It suffices to show that $f^Jf_J$ is $L^{p/2}$ on $A_{P,t}\times \omega$. As in the calculation in \cite[Proposition 5.5]{borel1974stable}, we have
\begin{equation*}
    |f^Jf_J|\prec a^{-2(\lambda+\alpha(J)+\gamma_j)}.
\end{equation*}
By \Cref{Lp-criterion of a function}, it is $L^{p/2}$ if $2\rho_P+p(\lambda+\alpha(J)+\gamma_j> 0$. By our assumption that $\sigma$ is weightless (resp. upper-middle), $\alpha(J)+\gamma_j\geq 0$ (resp. $\alpha(J)+\gamma_j>-\rho_P$), and therefore it holds true if $2\rho_P+p\lambda> 0$ (resp. $\lambda\geq 0$).
\end{proof}

\subsection{Integrability of Special Forms}\label{subsection: Integrability of Special Forms}
For any $1\leq p<\infty$, we denote by $A^\bullet_{(p)}(\E)$ the sheaf of $L^p$-forms on $\hX$, that is, the sheafification of the presheaf
\begin{equation*}
    U\in Op(\hX)\mapsto \{\sigma\ \textrm{a smooth form on}\ U\cap X|\textrm{Both}\ \sigma\ \textrm{and}\ d\sigma\ \textrm{are}\ L^p\}.
\end{equation*}
The goal of this subsection is to prove the following result.
\begin{Proposition}\label{Lp of special}
    \begin{itemize}
        \item [(i.)] For any $1\leq p<\infty$, there is an inclusion of sheaves
    \begin{equation*}
        \W^0C(\E)\hookrightarrow A_{(p)}(\E).
    \end{equation*}
        \item[(ii.)] there is an inclusion of sheaves
        \begin{equation*}
             \W^{-\rho+\varepsilon}C(\E)\hookrightarrow A_{(2)}(\E).
        \end{equation*}
    \end{itemize}
\end{Proposition}
\begin{proof}
    We only prove (i.), since the proof of (ii.) will be identical by applying the second part of \Cref{Lp-cr of a form}.
    
    It suffices to do it on the level of presheaves, and since a special form on an open subset is the restriction of a special form on $X$, and the differential of a special form is again special of the same weight profile, it suffices to show that any special form $\sigma$ on $X$ is $L^p$. Since $X$ has finite volume, we assume that $p\geq 2$ without loss of generality. Finally, by reduction theory, it suffices to show that the pullback to $D$ (still denoted by $\sigma$), is $L^p$ on any Siegel subset $S_{P,t,\omega}=\mu_P(A_{P,t}\times\omega)$. 

    Let $U_P$ be a geodesic neighborhood of $e_P$ in $D$ on which $\sigma$ is lifted from an $N_P$-invariant form on $e_P$. We note that there exists $t_0\geq t$ such that $S_{P,t_0,\omega}\subseteq V_P$ since $\omega$ is relatively compact, and Siegel sets form a fundamental system of neighborhoods of points in a corner (\cite[(6.2)]{borel1973corners}). Since $\sigma$ is lifted from $e_P\cong D_{L_P}\times N_P(\R)$, its pullback has constant growth along $A_P$ under the diffeomorphism $\mu_P:A_P\times D_{L_P}\times N_P(\R)\xrightarrow{\sim} D$, and since $\sigma$ is weightless, we apply \Cref{Lp-cr of a form} for $\lambda=0$ to see that $\sigma$ is $L^p$ on $S_{P,t_0,\omega}$ for all $2\leq p<\infty$. 

    Observe that
    \begin{equation*}
        S_{P,t,\omega}-S_{P,t_0,\omega}=\coprod_{P\subseteq Q}\mu_P(R_Q\times \omega),
    \end{equation*}
where
\begin{equation*}
    R_Q:=\{a\in A_{P,t}|a^\alpha\geq t_0\ \textrm{for}\ \alpha\in  \Delta_Q,a^\alpha<t_0\ \textrm{for}\ \alpha\notin \Delta_Q\}.
\end{equation*}
We claim each space in the disjoint union is contained in a Siegel set $S_{Q,t',\omega'}$ up to a relatively compact subset in $D$, such that $S_{Q,t',\omega'}$ is contained in a geodesic neighborhood $U_Q$ of $e_Q$. Then the above argument applies to each of these Siegel sets and completes the proof, since a globally defined smooth function on $D$ is $L^p$ on any relatively compact subset. The construction of these Siegel sets is a standard coordinate calculation via the geodesic neighborhood, and we leave it as the lemma below.
\end{proof}

Fix a parabolic subgroup $Q$ containing $P$. For any $a\in A_P$, we define $a_0,a_1\in A_P$ with 
\begin{equation*}
a_0^\alpha=\left\{
    \begin{aligned}
        &t_0,\alpha\in \Delta_Q\\
        &a^\alpha,\alpha\notin\Delta_Q
    \end{aligned}
    \right.,
    a_1^\alpha=\left\{
    \begin{aligned}
        &t_0^{-1}a^\alpha,\alpha\in\Delta_Q,\\
       & 1, \alpha\notin\Delta_Q.
    \end{aligned}
    \right.
\end{equation*}
Then $a=a_0a_1$. This gives a decomposition of $A_P$ with respect to the coordinates on $Q$-directions ($a_0$ stands for the "wall direction", and $a_1$ stands for the "ray direction"). We denote the set of all such $a_1$ appearing in the factorization of $a\in R_Q$ by $A_Q^+$, and the set of all $a_0$ appearing in the factorization of $a\in R_Q$ by $B$. 
Note that $a_1^\alpha=1$ when $\alpha\notin \Delta_Q$ while $a_1^\alpha\geq 1$ when $\alpha\in \Delta_Q$. Meanwhile, $t\leq a_0^\alpha<t_0$ for $a_0\in B$, so it is a relatively compact subset in $A_P$. We then have
\begin{equation*}
    \mu_P(R_Q\times \omega)=A_Q^+\cdot K,
\end{equation*}
where $K=\mu_P(B\times\omega)$ is relatively compact in $D$.  Set $\omega':=r_Q(K)\subseteq e_Q$, where $r_Q:D\to e_Q$ is the geodesic retraction. We shall consider Siegel sets of the form $S_{Q,t',\omega'}$.

Now we bring in the geodesic neighborhood. Consider a geodesic neighborhood of $e_Q$:
\begin{equation*}
    (r_Q,\rho_Q):U_Q\to e_Q\times[0,2)^{s}.\ (s=rank_\Q Q)
\end{equation*}

\begin{Lemma}
    When $t'$ is sufficiently large, $S_{Q,t',\omega'}$ is contained in $U_Q$, and its difference with $\mu_P(R_Q\times\omega)$ is relatively compact in $D$.
\end{Lemma}
\begin{proof}
    Note that the $A_Q^+$-action only increases $Q$-coordinates, and $K$ is relatively compact, there exists a uniform $c_0>0$ such that $A_{Q,c_0}^+\cdot K\subseteq U_Q$ where
\begin{equation*}
    A_{Q,c_0}^+:=\{a\in A^+_Q|a^\alpha\geq c_0\ \textrm{for} \ \alpha\in \Delta_Q\}.
\end{equation*}
The complement of this in $A^+_Q\cdot K$ is relatively compact in D. We replace $K$ by $K_1:=A_0\cdot K\subseteq U_Q$ where $A_0$ is the subset of $a^\alpha=c_0$ for $\alpha\in \Delta_Q$. Note that $\omega'=r_Q(K_1)$. Since $K_1$ is also relatively compact and is contained in $U_Q$, its distance (via $\rho_Q$) to $e_Q$ is bounded from below. In other words, there exists $\varepsilon>0$, such that $\omega'\times (0, \varepsilon)^{s}$ is disjoint from $K_1$. This is the Siegel set $S_{Q,t',\omega'}$ we are looking for (the distance $\varepsilon$ determines $t'>0$ in turn). The difference between this Siegel set and $\mu_P(R_Q\times \omega)=A_Q^+\cdot K$ consists of the part outside of $U_Q$, which is relatively compact; and the part in $U_Q$, which is contained in $\omega'\times [\varepsilon,1]^{s}$, is relatively compact as well.
\end{proof}

\subsection{An Eventual Quasi-isomorphism and Surjectivity}\label{subsection: An Eventual Quasi-isomorphism and Surjectivity}

The inclusion for weight profile $0$ established in the previous subsection becomes an eventual quasi-isomorphism.
\begin{Proposition}\label{eventual isomorphism}
    For $p<\infty$ sufficiently large (depending on $E$), the inclusion $\W^0C(\E)\hookrightarrow A_{(p)}(\E)$
    becomes a quasi-isomorphism.
\end{Proposition}
\begin{proof}
    It suffices to check the isomorphism on the stalk cohomology. The stalk cohomology of $\W^0C(\E)$, according to \cite[(17.2)]{goresky1994weighted}, at the stratum $X_P$, is computed by the truncation of the Lie algebra cohomology
    \begin{equation*}
        H^i(\W^0C(\E))|_{X_P}\cong \H^i(\n_P,E)_{\geq 0}.
    \end{equation*}
    Here $\H$ denotes the local system on $X_P$ associated to the $L_P$-representation $H^i(\n_P,E)$. It remains to determine the stalk of $L^p$-cohomology. In fact, it can be shown verbatim following \cite[(2.2.1),(3.1.4),(3.1.7)]{zucker2001reductive} that, when $p$ is sufficiently large, the stalk cohomology of $H^i(A_{(p)}(\E))$ at a point in $X_P$ has contributions only from those weights $\beta\in \Phi(\n_P,S_P)$ satisfying $-p\beta-2\rho_P\leq 0$ for arbitrarily large $p$, i.e., $\beta\geq 0$; and the precise contribution is identified with $\oplus_{\beta\geq0}H^i(\n_P,E)=H^i(\n_P,E)_{\geq 0}$.
\end{proof}

Finally, we discuss a consequence to the reduced $L^2$-cohomology by connecting weight profile $0$ and the middle weight profile $\rho$. As $X$ is complete as a Riemannian manifold, $L^2$-Stokes theorem holds and the reduced $L^2$-cohomology $\overline{H}^*_{(2)}(X,\E)$ is identified with the space $\H^*_{(2)}(X,\E)$ of $L^2$-harmonic forms. 

\begin{Theorem}\label{p to 2 surjectivity}
    The inclusion $\W^0C(\E)\hookrightarrow \W^{-\rho+\varepsilon}C(\E)\hookrightarrow A_{(2)}(\E)$ induces a surjection
    \begin{equation*}
        \W^0H^*(X,\E)\twoheadrightarrow \overline{H}^*_{(2)}(X,\E)
    \end{equation*}
\end{Theorem}
\begin{proof}
   Since $X$ has finite volume, we have an inclusion of sheaves
   \begin{equation*}
       A_{(p)}(\E)\hookrightarrow A_{(2)}(\E),
   \end{equation*}
   for any $p\geq 2$. We have the following commutative diagram of inclusions of sheaves
   \begin{equation*}
       \begin{tikzcd}
\W^0C(\E) \arrow[r, hook] \arrow[d, hook] & A_{(p)}(\E) \arrow[d, hook] \\
\W^{-\rho+\varepsilon}C(\E) \arrow[r]                 & A_{(2)}(\E)    .            
\end{tikzcd}
   \end{equation*}
   When $p$ is sufficiently large, the upper arrow becomes a quasi-isomorphism, and therefore it suffices to show that
   \begin{equation*}
       \W^0H^*(X,\E)\cong H^*_{(p)}(X,\E)\to \overline{H}^*_{(2)}(X,\E)
   \end{equation*}
   is surjective. By \cite[Lemma 3.10]{bergeron2006proprietes}, for every $L^2$-harmonic form $\omega$ on a Riemannian manifold of nonpositive and bounded below sectional curvature, there exists $c_\omega>0$, such that
   \begin{equation*}
       ||\omega||_{\infty}\leq c_\omega ||\omega||_{2}.
   \end{equation*}
   Since $X$ has finite volume, $L^\infty$-forms are $L^p$ for every $p>0$, and therefore $\omega\in A_{(p)}(\hX,\E)$.
\end{proof}
 
\begin{Remark}
    Via spectral decomposition, one can also show that 
    \begin{equation*}
            \im(\W^{-\rho+\varepsilon}H^*(X,\E)\to W^{-\rho} H^*(X,\E))=\H^*_{(2)}(X,\E),
    \end{equation*}
   see \cite[Proposition A.3]{nair2023automorphic}. One may rephrase the result above as 
   \begin{equation*}
       \H^*_{(2)}(X,\E)=\im (\W^0H^*(X,\E)\to \W^{-\rho}H^*(X,\E)).
   \end{equation*}
   However, to obtain this, it is necessary to modify our proof via replacing $A_{(2)}(\E)$ by some relative Lie algebra complexes as in \cite{nair2023automorphic}. In the Hermitian setting we concern, this is not necessary, via Zucker's conjecture.
\end{Remark}

\section{Algebraic Theory of Weighted Cohomology}\label{section: Algebraic Theory of Weighted Cohomology}
\subsection{Introduction}
In this section, let $D=G(\R)/KA_G$ be a \textit{connected Hermitian symmetric space}. We assume either that $G$ is semisimple, or that $G$ is reductive, while the Hermitian structure arises from a Shimura datum. Let $\Gamma\subseteq G(\Q)$ be a neat arithmetic subgroup, and $X=\Gamma\backslash G(\R)/KA_G$ be the associated connected Hermitian locally symmetric variety. Let $E$ be an irreducible representation of $G$ defined over $\Q$. The assumptions above provide a polarizable admissible variation of $\Q$-Hodge structures $\E^H$ on $X$ of certain weight $w\in \Z$, see \cite[\S 4]{zucker1981locally}, or \cite[1.4, 1.12, 1.13]{pink1989arithmetical} for these slightly different assumptions. Then it underlies a Hodge module $\E^H$ on $X$ of weight $w$.  

We denote by $\X$ its Baily--Borel compactification, which is a projective variety (\cite[Theorem 10.11]{baily1966compactification}). Then there is a canonical projection from the reductive Borel--Serre compactification $\pi:\hX\to \X$. We will discuss in more detail the Baily--Borel compactification in \Cref{subsection:BB}.

The main results of this section are due to Nair \cite[Theorem 4.3.1, Proposition 4.4.1]{nair2012mixed}, specializing to particular weights.

\begin{Theorem}\label{Theorem: Lie-Hodge correspondence}
\begin{itemize}
    \item [(i.)] There are canonical isomorphisms
    \begin{equation*}
        \pi_*\W^{-\rho+\varepsilon}C(\E)\cong \pi_*\W^{-\rho}C(\E)\cong \rat(IC_{\X}(\E^H));
    \end{equation*}
    \item [(ii.)] There is a canonical isomorphism
    \begin{equation*}
        \pi_*\W^0C(\E)\cong \rat(w_{\leq \underline{\dim}+w}IC_{\X}(\E^H));
    \end{equation*}
    \item [(iii.)] The isomorphisms above are compatible:
    \begin{equation*}
    \begin{tikzcd}
\pi_*\W^0C(\E) \arrow[r] \arrow[d]                   & \pi_*\W^{-\rho+\varepsilon}C(\E) \arrow[d] \\
\rat(w_{\leq \underline{\dim}+w}IC_{\X}(\E^H)) \arrow[r] & \rat(IC_{\X}(\E^H))       .                   
\end{tikzcd}
\end{equation*}
\end{itemize}

In particular, these weighted cohomology groups carry canonical mixed Hodge structures, and the connecting maps are morphisms of mixed Hodge structures.
\end{Theorem}

Together with the theory of weightless cohomology (\Cref{weightless and intersection}), we obtain the following result.
\begin{Corollary}\label{Image of MHS}
The canonical map $\W^0H^k(X,\E)\to IH^k(\X,\E)$ defined by the inclusion of weighted complexes
\begin{equation*}
    \W^{0}C(\E)\hookrightarrow \W^{-\rho+\varepsilon}C(\E)
\end{equation*}
is a morphism of mixed Hodge structures, and the image is identified with the top weight quotient $Gr^W_{k+w}\W^0H^k(X,\E)$.
\end{Corollary}

\begin{Corollary}\label{Injection}
    The canonical map $\W^0H^k(X,\E)\to IH^k(\X,\E)$ of weighted cohomology induces an injection of mixed Hodge structures
    \begin{equation*}
        Gr^W_{k+w}\W^0H^k(X,\E)\hookrightarrow IH^k(\X,\E).
    \end{equation*}
\end{Corollary}
 
 The key input is the correspondence, on each stratum, between the Lie algebra weights appearing in weighted cohomology and the Hodge-theoretic weights, via \textit{Looijenga's local Hecke operators}, see \Cref{subsection: local Hecke}. It is necessary to formalize our proof in the setting of stable $\infty$-categories rather than triangulated categories, see  \Cref{why infinity cats}.

Following a standard reduction as in \cite{looijenga19882} and \cite{goresky1994weighted}, We can and will assume in this section that $G$ is a reductive algebraic group such that $G^{ad}$ is $\Q$-simple, which will simplify the stratification of $\X$ and the description of the local Hecke operators.

\subsection{The Baily--Borel Compactification}\label{subsection:BB}
We follow the summary in \cite[22]{goresky1994weighted}, and one can find more detail in \cite[III]{ash2010smooth}. The description there for the Baily--Borel compactification $\X$ is only valid when $G^{ad}$ is $\Q$-simple.

The \textit{Baily--Borel compactification} of $X=\Gamma\backslash D$ is a normal projective variety with an open embedding $j:X\hookrightarrow \X$. As a topological space, $\X=\Gamma\backslash D^{BB}$, where $D^{BB}$ is set-theoretically the disjoint union of $D$ and its rational boundary components, equipped with the Satake topology. The action of $G(\Q)$ on $D$ extends to a continuous action on $D^{BB}$, and the stabilizer of each boundary component $F$ is a maximal proper rational parabolic subgroup $Q\subseteq G$ (when $F=D$, the stabilizer is $G$). 

Fix a minimal rational parabolic subgroup $P_0$ of $G$, and denote by $Q_1,Q_2,...Q_r$ the rational maximal parabolic subgroups of $G$ containing $P_0$. For each maximal parabolic subgroup $Q$, let $N_Q$ be the unipotent radical of $Q$ and $W_Q$ the center of $N_Q$. Then there is a total ordering $\prec$ on the set $\{Q_1,...,Q_r\}$, by
\begin{equation*}
    Q_i\prec Q_j\ \ \mathrm{iff}\ \  W_i\subseteq W_j\ \ \mathrm{iff}\ \  F_j\subseteq \overline{F_i}.
\end{equation*}
For convenience, we assume that $Q_0:=G\prec Q_1\prec ...Q_r$.

The more precise relation between a boundary stratum $F$ and the stabilizer $Q$ is as follows. The Levi factor $M:=Q/N_Q$ of $Q$ admits a decomposition as an almost direct product (commuting factors with finite intersections)
\begin{equation*}
    M=M_l\cdot A\cdot M_h,
\end{equation*}
where $A=A_Q\cong \G_m$ is the $\Q$-split center of $M$, and $M_h$ is the centralizer of the center $W=Z(U_Q)$ in $M$. The boundary stratum $F$ is then identified with $M_h(\R)/K_{M_h}$, the symmetric space of $M_h(\R)$. It is of Hermitian type. When $G^{ad}$ is $\Q$-simple, so is $M_h^{ad}$. 

For a neat arithmetic subgroup $\Gamma\subseteq G(\Q)$, there are neat arithmetic subgroups associated to the subgroups and quotients of $G$, either by taking the intersection or taking the quotient. Recall that (see \Cref{subsection: Special Forms and Weighted Cohomology})
\begin{equation*}
    \Gamma_Q:=Q(\R)\cap \Gamma_Q, \Gamma_{N_Q}:=N_Q(\R)\cap \Gamma, \Gamma_M:=\Gamma_Q/\Gamma_{N_Q}.
\end{equation*}
Furthermore, we define
\begin{equation*}
    \Gamma_{M_l}:= \Gamma_M\cap M_l(\Q), \Gamma_{M_h}:=\Gamma_M/\Gamma_{M_l}.
\end{equation*}
It turns out that the stratum $S$ of $\X$ covered by $F$ is then a smooth Hermitian locally symmetric space $S=\Gamma_{M_h}\backslash F= \Gamma_{M_h}\backslash M_h(\R)/K_{M_h}$. In this section, we shall use a coarser stratification as follows.

For each $d\in \{1,2,...,r\}$, let $D^{BB}_{\leq d}$ be the union of $D$ with the rational boundary components corresponding to maximal parabolic subgroups conjugate to $Q_i$ for some $i\leq d$. Let $\X_{\leq d}:=\Gamma\backslash D^{BB}_{\leq d}$. This defines a filtration
\begin{equation*}
    X=\X_{\leq0}\subseteq \X_{\leq 1}\subseteq...\subseteq \X_{\leq r}=\X.
\end{equation*}
by Zariski open dense subsets with $X_d:=\X_{\leq d}-\X_{\leq d-1}$ is a finite disjoint union of smooth Hermitian locally symmetric spaces $S$'s above. Then $\X=\coprod_{d=0,...,r}X_d$ is a stratification. We denote by $j_d:\X_{\leq d-1}\hookrightarrow \X_{\leq d}$ the open immersion, and $i_d:X_d\hookrightarrow \X_{\leq d}$ the closed complement. 

The relation to the reductive Borel--Serre compactification $\hX=\Gamma\backslash\widehat{D}$ (see \Cref{subsection: Special Forms and Weighted Cohomology}) is as follows. There is a \textit{unique} continuous projection $\widehat{D}\to \overline{D}^{BB}$ extending the identity map on $D$. The projection is $G(\Q)$-equivariant. Then it defines a projection 
\begin{equation*}
    \pi: \hX\twoheadrightarrow \X.
\end{equation*}
For each parabolic subgroup $P=Q_{i_1}\cap Q_{i_2}\cap...Q_{i_s}$ containing $P_0$, where $i_1<i_2<...<i_s$ the associated stratum $X_P\subseteq \hX$ projects to $S_{Q_{i_s}}\subseteq \X$. We call $Q$ the \textit{associated maximal parabolic subgroup} of $P$. We also define the filtration on $\hX$
\begin{equation*}
    X=\hX_{\leq 0}\subseteq \hX_{\leq 1}\subseteq...\subseteq \hX_{\leq r}=\hX,
\end{equation*}
where $\hX_{\leq d}:=\pi^{-1}(\X_{\leq d})$.

\subsection{Local Hecke Operators}\label{subsection: local Hecke}
In this subsection, we introduce the theory of local Hecke operators. First we will introduce, for each boundary stratum $S$ of $\X$, the local Hecke correspondence. This is a geometric correspondence of a neighborhood of $S$ in $\X$ leaving $S$ fixed pointwise. It will induce an endomorphism of the (Hodge) weight-truncation of the concerned sheaves on $S$. Then we will follow \cite[\S 15]{goresky1994weighted} to introduce the local Hecke operator on (Lie) weighted complexes. The direct image of the operator along the projection $\pi:\hX\to \X$ agrees with the one induced by the local Hecke correspondence.

\

$\bullet$ \textbf{Local Hecke correspondences.}

Let $S$ be a boundary stratum of $\X$, and $F$ be the corresponding stratum of $D^{BB}$ with stabilizer $Q$. Let $N_Q$ be the unipotent radical of $Q$, $M$ the Levi factor, and $W$ the center of $N_Q$. Let $A$ be the $\Q$-split center of $M$. By the classification of the irreducible Hermitian symmetric domains (\cite[Proposition III.2.8]{ash2010smooth}), $G^{ad}$ must be either type $BC_r$ or type $C_r$. One concludes from the root decomposition that 
\begin{itemize}
    \item [(i.)] The commutator $[N_Q, N_Q]\subseteq W$, therefore both $W$ and $V:=N_Q/W$ are vector-space groups, in particular, both are abelian;
    \item[(ii.)] The adjoint action of $A\cong \G_m$ on the lie algebra $Lie(W(\R))$ is given by a character $\chi^2\in X^*(A)$, and $A$ acts by $\chi\in X^*(A)_\Q$ on $Lie(V(\R))$.
\end{itemize}
\begin{Remark}
     When $G$ is $\Q$-simply connected, $\chi\in X^*(A)$ is integral.
\end{Remark}

Fix a lift $A\subseteq N_Q$. By \cite[3.6]{looijenga19882}, there exists $a\in A(\Q)$ such that $\chi(a)\in \Z_{>0}$ and $a\Gamma_Q a^{-1}\subseteq \Gamma_Q$. Such elements are called \textit{divisible}. The following lemma follows from the definition of the Satake topology on $D^{BB}$.

\begin{Lemma}\label{Good nbd of S}
    Fix a divisible element $a\in A(\Q)$. There exists a closed neighborhood $B_F$ of $F$ in $\textrm{Star}(F)$ such that
    \begin{itemize}
    \item [(i.)] $B_F$ is invariant under $N_Q(\R)$ and $\Gamma_Q$;
    \item [(ii.)] The restriction to $B_S:=\Gamma_Q\backslash B_F$ of the map $\Gamma_Q\backslash \textrm{Star}(F)\to \X$ is an embedding, identifying $B_S$ with a closed neighborhood of $S$ in $\textrm{Star}(S)$;
    \item [(iii.)] The map $\varphi(\Gamma_Qx):=\Gamma_Qax$ defines an endomorphism of $B_S$ fixing $S$ pointwise.
\end{itemize}
\end{Lemma}
Here the star of a stratum is the union of all strata containing it in their closure.
\begin{proof}
    For the first two statements see \cite[p.5]{looijenga19882}. The last statement follows by the fact that one can shrink $B_F$ so that it is invariant under the geodesic action of the semigroup $\{a\in A(\R)^0|\chi(a)\geq 1\}$ (see \cite [p.5 (b)]{looijenga19882}).
\end{proof}

\begin{Definition}\label{global Hecke correspondence}
For $g\in G(\Q)$ and $\Gamma'\subseteq \Gamma\cap g^{-1}\Gamma g$ of finite index, the \textit{Hecke correspondence} is the pair of finite morphisms
\begin{equation*}
\begin{aligned}
    (c_1, c_2): X':=\Gamma'\backslash D&\rightrightarrows X,\\
    c_1(\Gamma'x)&:=\Gamma x\\
    c_2(\Gamma'x)&:= \Gamma gx.
\end{aligned}
\end{equation*}
\end{Definition}

They extend uniquely to a finite correspondence $(c_1,c_2):X^{'BB}\rightrightarrows \X$,
which preserves the filtration $(c_1,c_2):X_{\leq d}^{'BB}\rightrightarrows \X_{\leq d}$ for each $d$.

We now define local Hecke correspondence. Let $S$ be a component boundary stratum corresponding to the maximal parabolic $Q$. Let $a\in A(\Q)$ be a divisible element, $\Gamma':=\Gamma\cap a^{-1}\Gamma a$. Since $a\Gamma_Q a^{-1}\subseteq \Gamma_Q$, we have $\Gamma'_Q=\Gamma_Q$.

\begin{Definition}\label{def: Local hecke correspondence}
     By \Cref{Good nbd of S}, the Hecke correspondence restricts to the neighborhood $B_S$ of $S$ a correspondence
\begin{equation*}
    (c_1, c_2): B_S\rightrightarrows B_S,
\end{equation*}
where $c_1$ is the identity map, and $c_2=\varphi$. We call it a \textit{local Hecke correspondence} of $S$ (with respect to $a$).
\end{Definition}

\

$\bullet$ \textbf{Local Hecke operators on the Hodge weight-truncations.}

Let $E$ be an irreducible representation of $G$ defined over $\Q$. Then the associated local system $\E$ on $X$ underlies a polarizable admissible variation of $\Q$-Hodge structures $\E^H$. For a Hecke correspondence $(c_1, c_2):X_{\Gamma'}\rightrightarrows X_{\Gamma}$, we have $c_1^!\E^H_\Gamma=\E^H_{\Gamma'}=c_2^*\E^H_{\Gamma}$. 

We denote by $J_{d-1}:X\hookrightarrow \X_{\leq d-1}$ and $j_d:\X_{\leq d-1}\hookrightarrow \X_{\leq d}$ be the open immersions. Then we obtain a morphism of variation of mixed Hodge structures on $\X_{\leq d-1}$:
\begin{equation*}
    c_{1*}c_2^*J_{d-1,*}\E^H\to J_{d-1,*}\E^H
\end{equation*}
via the adjunction of
\begin{equation*}
    c_2^*J_{d-1,*}\E^H\to J'_{d-1,*}c_2^*\E^H=J'_{d-1,*}c_1^!\E^H\to c_1^!J_{d-1,*}\E^H,
\end{equation*}
where the two arrows are given by the base change. 

\begin{Proposition}\label{Hecke correspondence-Hodge truncation}
Let $\aa=(a_0,...a_{d-1})\in \Z^d$. Denote by $K_{d-1}:=w_{\leq \aa}J_{d-1,*}\E^H$ the variation of mixed Hodge structures on $\X_{\leq d-1}$.
Then there is a unique map $c_{1*}c_2^*K_{d-1}\to K_{d-1}$ such that the following diagram commutes
    \begin{equation*}
        \begin{tikzcd}
c_{1*}c_2^*K_{d-1} \arrow[r] \arrow[d] & K_{d-1} \arrow[d] \\
{c_{1*}c_2^*J_{d-1,*}\E^H} \arrow[r]   & {J_{d-1,*}\E^H}  .
\end{tikzcd}
    \end{equation*}
\end{Proposition}
\begin{proof}(sketch.)
    This is \cite[Lemma 5.1.3]{morel2008complexes}. It requires the following facts: 
    \begin{itemize}
        \item [$\bullet$] $c_2^*F\in \wD{\leq \aa}$ for $F\in \wD{\leq \aa}$ and $c_1^!G\in \wD{>\aa}$ for $G\in \wD{> \aa}$. This holds since both $c_1$ and $c_2$ are finite maps of stratified spaces;
        \item [$\bullet$] For $F\in \wD{\leq \aa}$ and $G\in \wD{> \aa}$, one has $R\Hom(F, G)=0$. This is because $(\wD{\leq \aa}, \wD{>\aa})$ is a t-structure stable under shifts.
    \end{itemize}
   Applying these two facts to the weight truncation of $J_{d-1,*}\E^H$, one obtains the the map and the uniqueness follows.
\end{proof}

Let $i_S:S\hookrightarrow X_d=\X_{\leq d}-\X_{\leq d-1}$ be a component stratum. Fix a divisible element $a\in A(\Q)$, and let $B_S\subseteq \X_{\leq d}$ be a neighborhood of $S$ as in \Cref{Good nbd of S}. Consider the Hecke correspondence with respect to $a$, and the map $c_{1*}c_2^*K_{d-1}\to K_{d-1}$ defined in \Cref{Hecke correspondence-Hodge truncation}. Since $c_1$ reduces to the identity map on $B_S$ and $c_2=\varphi$, this map induces a canonical map
\begin{equation}\label{phi-equivariant, Hodge}
\varphi^*(j_{d*}K_{d-1}|_{B_S})\to j_{d*}K_{d-1}|_{B_S}.
\end{equation}
Strictly speaking, this is the composition of a canonical map
\begin{equation*}
    \varphi^*(j_{d*}K_{d-1}|_{B_S})\to j_{d*}c_{1*}c_2^*K_{d-1}|_{B_S}
\end{equation*}
obtained by base changes, and 
\begin{equation*}
    j_{d*}c_{1*}c_2^*K_{d-1}|_{B_S}\to j_{d*}K_{d-1}|_{B_S}
\end{equation*}
induced by the map in \Cref{Hecke correspondence-Hodge truncation}.

\begin{Definition}\label{Local Hecke: Hodge}
    For $a\in A(\Q)$ divisible, we define the \textit{local Hecke operator} $\Phi$ as the endomorphism on $i_S^*j_{d*}K_{d-1}$ to be the restriction to $S$ of the map (\ref{phi-equivariant, Hodge}). 
\end{Definition}

This makes sense since $\varphi|_S=id$.

\

$\bullet$ \textbf{Local Hecke operators on weighted complexes.}

Now we turn to the Lie-theoretic side. Let $S\in X_{d}=\X_{\leq d}-\X_{\leq d-1}$ be a component stratum, corresponding to a maximal parabolic subgroup $Q$.  Let $X_Q$ be the boundary stratum in $\hX$. It is then a maximal stratum in $\hX_{\leq d}-\hX_{\leq d-1}$, and its closure in $\hX_{\leq d}$ is $\pi^{-1}(S)$. Let $\widehat{j_d}:\hX_{\leq d-1}\hookrightarrow \hX_{\leq d}$ be the open immersion.

For a weight profile $\eta$ (see \Cref{subsection: Parabolic Subgroups and Weight Profiles}), let 
\begin{equation*}
    L_{d-1}:=\pi_*\W^{\eta}C(\E)|_{\X_{\leq d-1}}.
\end{equation*}
For any $g\in G(\Q)$, It is not hard to show that there is a canonical morphism $c_{1*}c_2^*L_{d-1}\to L_{d-1}$ similar to \Cref{Hecke correspondence-Hodge truncation}. For a lift $A\subseteq N_Q$, and a divisible element $a\in A(\Q)$, one can define similarly a local Hecke operator $\Phi$ on $i_S^*j_{d*}L_{d-1}$. There is a useful characterization of this local Hecke operator given in \cite[\S 15, \S 16]{goresky1994weighted}, which we summarize in the following way.

Let $\E_\C$ be the complexification of $\E$. Consider the sheaf 
\begin{equation*}
    B:=\widehat{j_d}_*^0(\W^\eta C(\E_\C)|_{\hX_{\leq d-1}}).
\end{equation*}

This underived pushforward sheaf represents the derived pushforward $\widehat{j_d}_*(\W^\eta C(\E_\C)|_{\hX_{\leq d-1}})$ in the derived category. Similar to the construction of special differential forms (see \Cref{subsection: Special Forms and Weighted Cohomology}), there exists a quasi-isomorphic subcomplex of special differential forms $B_{sp}\subseteq B$, see \cite[(16.8)]{goresky1994weighted}. Moreover, there is an action of $A(\Q)$ on the restriction to some neighborhood of $\pi^{-1}(S)$ in $\hX_{\leq d}$ of the sheaf $B^{sp}$, via the action on the Koszul complexes, see \cite[(15.2), (16.12)]{goresky1994weighted}. When restricting to $\pi^{-1}(S)$, this action agrees with local Hecke operators in the sense of \cite[(15.4), (15.5)]{goresky1994weighted}. (Note that on the reductive Borel--Serre compactification, one allows more elements than divisible ones to act on the sheaves).

Fix a divisible element $a\in A(\Q)$. Then there is an endomorphism of the sheaf
    \begin{equation*}
    i_S^*j_{d*}L_{d-1,\C}\cong\pi_*(B^{sp}|_{\pi^{-1}(S)})
    \end{equation*}
by the direct image of the action of $a$ restricting to $B_{sp}|_{\pi^{-1}(S)}$.

Since the strata in $\pi^{-1}(S)$ are of the form $X_P$ for $P\subseteq Q$, hence $A(\Q)\subseteq N_Q(\Q)\subseteq N_P(\Q)$, divisibility of $a$ implies that
    \begin{equation*}
        \Gamma'_P:=a^{-1}\Gamma_Pa=\Gamma_P.
    \end{equation*}
By \cite[(15.5)]{goresky1994weighted}, this action agrees with local Hecke operator $\Phi$ arising from the local Hecke correspondence near $S$.

\subsection{Proof of the Main Result}
In this subsection we prove \Cref{Theorem: Lie-Hodge correspondence}. The first part is \cite[(23.2)]{goresky1994weighted}, therefore it suffices to prove (ii.) and (iii.). We restrict ourselves to the weight profile $\eta=0$, even though the proof is valid verbatim for any weight profile; In particular, the proof is valid verbatim for (i.). 

We first summarize the notations from the previous subsections, and then outline the proof. 

Let 
\begin{equation*}
    X=\X_{\leq0}\subseteq \X_{\leq 1}\subseteq...\subseteq \X_{\leq r}=\X.
\end{equation*}
be the filtration on $\X$ defined in \Cref{subsection:BB}. For each $d\in \{0,1,...,r\}$, let $j_d:\X_{\leq d-1}\hookrightarrow \X_{\leq d}$ be the open immersion, and $i_d:X_d=\X_{\leq d}-\X_{\leq d-1}\hookrightarrow \X_{\leq d}$ be the closed complement. Let $S\subseteq X_d$ be a component stratum, corresponding to a maximal parabolic subgroup $Q$. Denote by $i_S:S\hookrightarrow \X_{\leq d}$ the closed immersion. 

Set
\begin{equation*}
\begin{aligned}
    L_d&:=\pi_*\W^0C(\E)|_{\X_{\leq d}},\\
    K_d&:=w_{\leq \underline{\dim}+w}j_*\E^H|_{\X_{\leq d}},\\
    L_S&:=i_S^*j_{d*}L_{d-1},\\
    K_S&:=i_S^*j_{d*}K_{d-1}.  
\end{aligned}
\end{equation*}
Let $A$ be a lift of the split center of $Q$, and fix a divisible element $a\in A(\Q)$. Then we obtain local Hecke operators $\Phi$ on $i_S^*L_d$ and $i_S^*K_d$. 

Finally, for any $g\in G(\Q)$ and $\Gamma'\subseteq\Gamma\cap g^{-1}\Gamma g$ of finite index, there is a Hecke correspondence 
\begin{equation*}
    (c_1,c_2): X_{\leq d}^{'BB}:=(\Gamma'\backslash D)^{BB}_{\leq d}\rightrightarrows \X_{\leq d}.
\end{equation*}
Sheaves on $X_{\leq d}^{'BB}$ will be decorated with a prime $'$, e.g., $L_d', K_d'$.

We will prove by induction on $d$ for the following statements:

\begin{itemize}
    \item [(i.)$_{d}$] There is a canonical isomorphism $\alpha_d:L_d\cong \rat(K_d)$;

\item [(ii.)$_{d}$] The isomorphism $\alpha_d$ is compatible with translations: There is a commutative diagram 
\begin{equation*}
    \begin{tikzcd}
c^*L_d \arrow[d, "c^*(\alpha_d)"] \arrow[r] & L_d' \arrow[d, "\alpha'_d"] \\
\rat(c_2^*K_d) \arrow[r]                              & \rat(K_d'),   
\end{tikzcd}
\end{equation*}
for any $g\in G(\Q)$ and $\Gamma'\subseteq\Gamma\cap g^{-1}\Gamma g$ of finite index, here $c:=c_2$;

\item [(iii.)$_d$] The isomorphism $\alpha_d$ is Hecke equivariant: There is a commutative diagram 
\begin{equation*}
    \begin{tikzcd}
c_{1*}c_2^*L_d \arrow[d, "c_{1*}c_2^*(\alpha_d)"] \arrow[r] & L_d \arrow[d, "\alpha_d"] \\
\rat(c_{1*}c_2^*K_d) \arrow[r]                              & \rat(K_d);   
\end{tikzcd}
\end{equation*}

\item [(iv.)$_d$] The isomorphism $\alpha_d$ is weight compatible: There is a commutative diagram
\begin{equation*}
\begin{tikzcd}
L_d \arrow[r] \arrow[d, "\alpha_d"] & (\pi_*\W^{-\rho+\varepsilon}C(\E))|_{\X_{\leq d}} \arrow[d, "\beta_{d}"] \\
\rat(K_d) \arrow[r]                 & IC(\E)|_{\X_{\leq d}}    ,           \end{tikzcd}
\end{equation*}
where $\beta: \pi_*\W^{-\rho+\varepsilon}C(\E)\cong IC(\E)$ is the canonical isomorphism in (i.) of \Cref{Theorem: Lie-Hodge correspondence}.
\end{itemize}

The induction step goes as follows. First, the result in \cite{goresky1994weighted} implies that $L_{S,\C}$ admits a $\Phi$-invariant direct sum decomposition such that $i_S^*L_{d,\C}$ is a direct summand as a $\Phi$-invariant subspace. On the other side, the result in \cite{looijenga1991weights} implies that the Hodge weights in $\rat(K_S)$ correspond to the weights of $\Phi$. Together with the induction step (i.)$_{d-1}$ and (iii.)$_{d-1}$, one can show that the Lie weights match with Hodge weights, which implies (i.)$_d$. Second, (i.)$_d$, (ii.)$_{d-1}$ and (iii.)$_{d-1}$ imply (ii.)$_d$. Third, (ii.)$_d$ is an intermediate step to (iii.)$_d$. Finally, (i.)$_d$ and (iii.)$_{d-1}$ imply (iv.)$_d$. In Step 1 below, we follow the strategy in \cite{nair2012mixed}, with a slightly different narration. Step 2 and Step 3 differ from Nair's original proof.

\

$\bullet$ \textbf{Setup.}

We start with $d=0$, by choosing an isomorphism $\alpha_0:L_0=\E\cong \rat(\E^H)=K_0$ on $X$. Such an isomorphism is unique up to a scalar. It is Hecke equivariance, since both Hecke actions come from the same representation $E$. 

Assume now that $d\geq 1$ and (i.)$_{d-1}$-(iv.)$_{d-1}$ hold.

\

$\bullet$ \textbf{Step 1: Local weights correspondences.}

We begin with the weighted complex side. The following proposition is the splitting of the weighted complexes, see \cite[Proposition 16.4]{goresky1994weighted}.

\begin{Proposition}\label{Splitting Property of weighted complex}
There is a $\Phi$-invariant direct sum decomposition of 
\begin{equation*}
    L_{S,\C}\cong\bigoplus_{\lambda\in \Q} (L_{S,\C})_\lambda,
\end{equation*}
such that $\Phi$ acts as rational scalar $\lambda$ on each component. The exact triangle
\begin{equation*}
    i_S^*L_{d,\C}\to L_{S,\C}\to i_S^*i_{d*}i_d^!L_{d,\C}[1]
\end{equation*}
is split and isomorphic to
\begin{equation*}
    L_{S,\C,\leq}:=\bigoplus_{\lambda\leq 1}(L_{S,\C})_\lambda\to L_{S,\C}\to L_{S,\C,>}:=\bigoplus_{\lambda>1} (L_{S,\C})_\lambda
\end{equation*}
\end{Proposition}

\begin{proof}(sketch.)
Via the concrete representative $B_{sp}$ of $\widehat{j_d}_*(\W^\eta C(\E_\C)|_{\hX_{\leq d-1}})$ (see \Cref{subsection: local Hecke}), it is shown in \cite[\S 16]{goresky1994weighted} that there is an honest direct sum decomposition of sheaves on $\pi^{-1}(S)$ (\cite[16.13]{goresky1994weighted})
\begin{equation*}
    B_{sp}|_{\pi^{-1}(S)}=\bigoplus_{\lambda}(B_{sp}|_{\pi^{-1}(S)})_{\lambda},
\end{equation*}
    such that the local Hecke operator acts by the scalar $\lambda$ on each component. It follows that these eigenvalues $\lambda$ are rational numbers, since they are all values of certain rational characters of $A$ at the divisible element $a\in A(\Q)$. Moreover, the weighted complex is a subsheaf of $B_{sp}$ as the direct summand (\cite[16.11]{goresky1994weighted})
    \begin{equation*}
        \W^0C(\E_\C)|_{\pi^{-1}(S)}=\bigoplus_{\lambda\leq \chi(a)^0=1} (B_{sp}|_{\pi^{-1}(S)})_{\lambda}\subseteq B_{sp}|_{\pi^{-1}(S)}.
    \end{equation*}
This honest decomposition of sheaves defines a $\Phi$-equivariant direct sum decomposition in the stable $\infty$-category $D^b_c(\C_{\pi^{-1}(S)})$. Applying $\pi_*$ we obtain the decomposition of $\pi_*(B_{sp}|_{\pi^{-1}(S)})\cong L_{S,\C}$
with respect to $\Phi$
\begin{equation*}
    L_{S,\C}\cong \bigoplus_{\lambda\in \Q}(L_{S,\C})_\lambda,
\end{equation*}
    with 
    \begin{equation*}
        i_S^*L_{d,\C}\cong \pi_*(\W^0C(\E_\C)|_{\pi^{-1}(S)})\cong \bigoplus_{\lambda\leq 1}(L_{S,\C})_\lambda.
    \end{equation*}
\end{proof}

The following is a direct corollary to \cite[Corollary 16.5]{goresky1994weighted}.
\begin{Corollary}\label{Splitting: Global section}
    For any subset $Z\subseteq S$ the action of $\Phi$ on the hypercohomology $\mathbb{H}^i(Z, j_{d*}L_{d-1})$ is semisimple over $\Q$, and applying $\mathbb{H}^i(Z,-)$ to the exact triangle
    \begin{equation*}
        L_{d}\to j_{d*}L_{d-1}\to i_{d*}i_d^!L_d[1]
    \end{equation*}
    gives a short exact sequence
    \begin{equation*}
        0\to \mathbb{H}^i(Z, j_{d*}L_{d-1})_{\lambda\leq 1}\to \mathbb{H}^i(Z, j_{d*}L_{d-1})\to \mathbb{H}^i(Z, j_{d*}L_{d-1})_{\lambda>1}\to 0.
    \end{equation*}
\end{Corollary}

\begin{proof}
We obtain the $\Phi$-semisimplicity and the short exact sequence above after tensor with $\C$ by \cite[(16.5)]{goresky1994weighted}. Since the eigenvalues are rational numbers by the divisibility of $a$, the claims hold over $\Q$.
\end{proof}

We prove that the decomposition in \Cref{Splitting Property of weighted complex} is defined over $D^b(\Q_S)$, using the following general lemma. It is a generalization of \cite[Lemme 3.2.5]{laumon2008lemme}. 

\begin{Lemma}\label{Jordan Decomposition}
   Let $E$ be a field, and $T$ be a discrete abelian group. Let $\CC$ be an $E$-linear stable $\infty$-category with a $t$-structure with heart $\AA=\CC^{\heartsuit}$. Set
   \begin{equation*}
       \CC^{T}:=Fun(BT,\CC),\ \  \AA^{T}:=Fun(BT,\AA).
   \end{equation*}
   Equip $\CC^{T}$ with its pointwise t-structure, so that $\AA^T\simeq \CC^{T,\heartsuit}$. For a character $\chi\in \Hom(T,E^{\times})$, let $\CC^{T}_{\chi}\subseteq \CC^{T}$ be the full subcategory consisting of objects $X$ such that for any $\gamma\in T$, $\gamma-\chi(\gamma)$ acts nilpotently on $X$. Likewise we define $\AA^{T}_{\chi}$. 

Let $(\CC^{T}_{fs})^b\subseteq \CC^T$ be the full subcategory of bounded $K$, such that for each $n\in \Z$, 
\begin{equation*}
    H^n(K)\in \bigoplus_{\chi}\AA^T_{\chi}.
\end{equation*}
Then there is a t-exact equivalence of categories
\begin{equation*}
    \bigoplus_{\chi}(\CC^{T}_\chi)^{b}\xrightarrow{\sim} (\CC^{T}_{fs})^b.
\end{equation*}
   
\end{Lemma}

\begin{proof}(sketch.)
We only outline the proof, since each step is a formal verification.

(1.) $\CC^T_\chi$ is a stable full subcategory preserved by the pointwise truncations, and thus $(\CC^T_\chi)^b$ inherits a t-structure.

(2.) For $\chi\neq \chi'$, the mapping spectrum
    \begin{equation*}
     hom_{\CC^T}(X,Y)\simeq 0, \ \ X\in \CC^T_{\chi},\ \ Y\in \CC^T_{\chi'}.
    \end{equation*}
To see this, let $\gamma\in T$ be an element $\chi(\gamma)\neq \chi'(\gamma)$. Suppose that $(\gamma-\chi(\gamma))^p$ (resp. $(\gamma-\chi'(\gamma))^q$) kills $X$ (resp. $Y$). By Bézout's theorem, there exists $P(t), Q(t)\in E[t]$ such that
\begin{equation*}
    (t-\chi(\gamma))^p P(t)+(t-\chi'(\gamma))^qQ(t)=1.
\end{equation*}
Let $f\in \Hom_{\CC^T}(X,Y)$. Then 
\begin{equation*}
    f=(\gamma-\chi(\gamma))^p P(\gamma)f+f(\gamma-\chi'(\gamma))^qQ(\gamma)\simeq 0.
\end{equation*}
Since $\CC^T_\chi$ is stable under shifts, the mapping spectrum vanishes as well. 

(3.) Orthogonality in (2.) establishes
    \begin{equation*}
        hom_{\CC^T}(\bigoplus_{\chi}X_\chi,\bigoplus_{\chi'}Y_{\chi'})\simeq \prod_{\chi}hom_{\CC^T}(X_\chi, Y_\chi).
    \end{equation*}
This shows the fully faithfulness of the functor.

(4.) Essential surjectivity follows by a Postnikov induction. Let $K\in (\CC^T_{fs})^b$. The finite Postnikov tower of $K$ has layers of the form $H^n(K)[-n]$. By orthogonality (2.), all connecting morphisms are block diagonalizable with respect to characters $\chi$, and the tower reconstructs $K$ in each character block by (1.), which gives a decomposition (unique up to a contractible choice)
\begin{equation*}
    K\simeq \bigoplus_\chi K_\chi,\ K_\chi\in (\CC^T_\chi)^b,\ \ H^n(K_\chi)\simeq H^n(K)_\chi.
\end{equation*}

(5.) Taking cohomology commutes with finite direct sums, thus
\begin{equation*}
    H^n(K)\simeq H^n(\bigoplus_\chi K_{\chi})\simeq \bigoplus_\chi H^n(K_\chi)\simeq \bigoplus_\chi H^n(K)_\chi,
\end{equation*}
    hence the t-exactness.
\end{proof}

\begin{Corollary}\label{splitting is defined over Q}
    The $\Phi$-invariant decomposition of $L_{S,\C}$ in \Cref{Splitting Property of weighted complex} is defined over $\Q$:
    \begin{equation*}
        L_S\cong\bigoplus_{\lambda\in \Q}(L_S)_{\lambda},
    \end{equation*}
    with compatible isomorphisms
    \begin{equation*}
    \begin{aligned}
        i_S^*L_d&\cong L_{S,\leq}:=\bigoplus_{\lambda\leq 1}(L_{S})_\lambda,\\
        i_S^*i_{d*}i_d^!L_d[1]&\cong L_{S,>}:=\bigoplus_{\lambda>1}(L_S)_{\lambda}.
    \end{aligned} 
\end{equation*}
\end{Corollary}

\begin{proof}
By \Cref{Splitting: Global section}, for any $x\in S$, the action of $\Phi$ on
    \begin{equation*}
        H^n(L_S)_x=H^n(i_x^*L_S)=\mathbb{H}^n(\{x\}, j_{d*}L_{d-1})
    \end{equation*}
is semisimple over $\Q$. Since $S$ is a connected stratum, the cohomology $H^n(L_S)$ is then a local system of finite rank, and the eigenvalues of the $\Phi$-action on it are constant along $S$. Therefore we obtain sub-local systems 
\begin{equation*}
    H^n(L_S)_\lambda:=\ker(\Phi-\lambda: H^n(L_S)\to H^n(L_S)), \ \lambda\in \Q.
\end{equation*}
Fiberwise semisimplicity then gives a decomposition
\begin{equation*}
    H^n(L_S)=\bigoplus_{\lambda} H^n(L_S)_\lambda.
\end{equation*}
Applying \Cref{Jordan Decomposition} to $E=\Q$, $\CC=D^b_c(\Q_S)$, and $T=\Phi^\Z\cong \Z$, we obtain a decomposition
\begin{equation*}
    L_S\simeq \bigoplus_\lambda(L_S)_\lambda.
\end{equation*}
Since the decomposition is unique, the complexification of this decomposition has to agree with that one obtained in \Cref{Splitting Property of weighted complex}. The other two identifications follow by the same reasoning.
\end{proof}

Now we turn to the Hodge side. By induction hypothesis, we obtain the following decomposition of $K_S$ in the derived category of mixed Hodge modules. 

\begin{Corollary}\label{Decomposition: Hodge side}
    There is a direct sum decomposition in $D^b\MHM(S)$
    \begin{equation*}
        K_S\cong\bigoplus_{\lambda\in \Q} (K_S)_\lambda,
    \end{equation*}
where $\Phi-\lambda$ acts nilpotently on $(K_S)_\lambda$. Such a decomposition is unique (up to a contractible choice).
\end{Corollary}
    
\begin{proof}
    By induction hypothesis (i.)$_{d-1}$ and (iii.)$_{d-1}$, there is a Hecke equivariant isomorphism $\alpha_{d-1}:L_{d-1}\xrightarrow{\sim} \rat(K_{d-1})$. This gives a $\Phi$-equivariant isomorphism
    \begin{equation*}
        L_S\xrightarrow{\sim} \rat(K_S).
    \end{equation*}
   Since $H^n(K_S)$ is a smooth Hodge module, we have
   \begin{equation*}
       \rat(H^n(K_S))\cong H^{n-\dim(X_d)}(\rat(K_S))\cong H^{n-\dim(X_d)}(L_S).
   \end{equation*}
   We obtain a decomposition as in the proof of \Cref{splitting is defined over Q}
   \begin{equation*}
       \rat(H^n(K_S))=\bigoplus_{\lambda}\rat(H^n(K_S))_\lambda,
   \end{equation*}
   where $\Phi$ acts as the scalar $\lambda$ on each component. Let $\Lambda_n\subseteq \Q$ be the finite set of eigenvalues on $\rat(H^n(K_S))$, and set
   \begin{equation*}
       P_n(t):=\prod_{\lambda\in \Q}(t-\lambda)\in \Q[t],
   \end{equation*}
   then $P_n(\rat(\Phi))=0$ on $\rat(H^n(K_S))$. Since $\rat$ is faithful on $\MHM(S)$, we have $P_n(\Phi)=0$ on $H^n(K_S)$. Therefore the projectors $ p_\lambda:=\prod_{\lambda\neq\mu\in\Lambda_n}\frac{\Phi-\mu}{\lambda-\mu}$ give a direct sum decomposition $H^n(K_S)=\oplus_{\lambda}H^n(K_S)_\lambda$, where
   \begin{equation*}
       H^n(K_S)_\lambda:=\im(p_\lambda).
   \end{equation*}
   Applying \Cref{Jordan Decomposition} we obtain the decomposition of $K_S$.
\end{proof}

\begin{Remark}
    We cannot apply \Cref{splitting is defined over Q} directly, since $\rat$ is only faithful on $\MHM(S)$ but it is not faithful on $D^b\MHM(S)$.
\end{Remark}

We are going to match Hodge weights with the weights of local Hecke operators. The following lemma is a generalization of \cite[Proposition 6.4]{looijenga1991weights}.

\begin{Proposition}\label{Cohomology weight decomposition}
    For each $n\in\Z$, applying the functor $\rat\circ H^n\cong {}^pH^{n+\dim(X_{d})}\circ\rat$ to the exact triangle in $D^b\MHM(S)$
    \begin{equation*}
        i_S^*K_d\to K_S\to w_{>\dim(X_d)+w}K_S,
    \end{equation*}
    one obtains a short exact sequence of local systems on $S$
    \begin{equation*}
       0\to H^n(\rat(K_S))_{\leq 1}\to H^n(\rat(K_S))\to H^n(\rat(K_S))_{\geq 1}\to 0,
    \end{equation*}
    given by the decomposition of $H^n(\rat(K_S))$ with respect to the eigenvalues of $\Phi$.
\end{Proposition}

\begin{proof}(sketch.)
Let $q=\chi(a)^2\in \Z$. Let $B_S$ be a neighborhood of $S$ as in \Cref{Good nbd of S} where the local Hecke correspondence $\varphi$ acts on. As shown in \cite[(6.4)]{looijenga1991weights}, there is a $q$-Hodge module structure (\cite[(2.2)]{looijenga19882}) on the full direct image of $\E^H$, after restricting to $B_S$. This definition makes sense for the Hodge weight truncations, since pointwise weight truncations define a $t$-structure on the derived category of $q$-mixed Hodge modules. It follows that, as in \cite[(6.4)]{looijenga1991weights}, $q^{w/2}\Phi$ defines a $q$-endomorphism (\cite[(1.1)]{looijenga1991weights}) of the stalk of the mixed Hodge structure $H^n(i_x^*K_S)$ for any point $i_s:x\hookrightarrow S$. Since the subspace of $\Phi$-eigenvalues $\leq 1$ in $H^n(i_x^*K_S)$ amounts to the subspace of $q^{w/2}\Phi$-eigenvalues $\leq q^{w/2}$, it follows from the definition of a $q$-endomorphism that 
\begin{equation*}
    H^n(i_x^*K_S)_{\leq 1}=W_{\leq w}H^n(i_x^*K_S).
\end{equation*}
On the other side, by \Cref{weight truncation for smooth MHM}, $W_{\leq w}H^n(i_x^*K_S)$ is exactly 
\begin{equation*}
    H^n(i_x^*w_{\leq \dim(X_d)+w}K_S)\cong H^n(i_x^*i_S^*K_d).
\end{equation*}
Taking $\rat$ gives the desired short exact sequence.
\end{proof}

\begin{Corollary}\label{Hodge-Looijenga decomposition}
The exact triangle in $D^b\MHM(S)$
\begin{equation*}
    i_S^*K_d\to K_S\to w_{>\dim(X_d)+w}K_S
\end{equation*}
    is isomorphic to the split exact triangle
    \begin{equation*}
        K_{S,\leq}\to K_S\to K_{S,>},
    \end{equation*}
where
\begin{equation*}
    K_{S,\leq}:=\bigoplus_{\lambda\leq 1}(K_S)_\lambda,\ \ K_{S,>}:=\bigoplus_{\lambda>1} (K_S)_\lambda
\end{equation*}
    from the decomposition in \Cref{Decomposition: Hodge side}.
\end{Corollary}

\begin{proof}
    By \Cref{Cohomology weight decomposition}, the morphism $K_{S,\leq}\to K_S$ induces isomorphisms on cohomology
    \begin{equation*}
        H^n(K_{S,\leq})\xrightarrow{\sim} H^n(i_S^*K_d)
    \end{equation*}
    for each $n\in \Z$. Hence the morphism factors through an isomorphism $K_{S,\leq}\xrightarrow{\sim}i_S^*K_d$. Taking the cofiber gives the isomorphism between two exact triangles.  
\end{proof}

\

$\bullet$ \textbf{Step 2: Construction of $\alpha_d$.}

We begin with the following diagram of exact triangles
\begin{equation}\label{Construct isomorphism}
\begin{tikzcd}
L_d \arrow[r] \arrow[d, "\alpha_d", dashed] & j_{d*}L_{d-1} \arrow[d, "j_{d*}\alpha_{d-1}"] \arrow[r] & {i_{d*}i_d^!L_d[1]}                             \\
\rat(K_d) \arrow[r]                         & \rat(j_{d*}K_{d-1}) \arrow[r]                           & i_{d*}\rat(w_{>\dim(X_d)+w}i_d^*j_{d*}K_{d-1}).
\end{tikzcd}
\end{equation}

Let $S\subseteq X_d$ be any component stratum. Applying $i_S^*$ to the diagram above, by \Cref{splitting is defined over Q} and \Cref{Hodge-Looijenga decomposition}, it becomes
\begin{equation*}
\begin{tikzcd}
{L_{S,\leq}} \arrow[r]  & {L_{S}} \arrow[d, "i_S^*j_{d*}(\alpha_{d-1})"] \arrow[r] & {L_{S,>}}       \\
{\rat(K_{S,\leq})} \arrow[r]             & {\rat(K_{S})} \arrow[r]                                  & {\rat(K_{S,>})}
\end{tikzcd}
\end{equation*}

All arrows above are $\Phi$-equivariant by induction hypothesis (ii.)$_{d-1}$, and therefore they lift to morphisms in the $\Phi$-equivariant category $D^b_{c}(\Q_S)^\Phi$. By \Cref{Jordan Decomposition}, the mapping space 
\begin{equation*}
    \Hom_{D^b_{c}(\Q_S)^\Phi}(L_{S,\leq}, \rat(K_{S,>}))\simeq 0
\end{equation*}
is contractible. Hence we obtain an isomorphism of mapping spaces
\begin{equation*}
    \Hom_{D^b_{c}(\Q_S)^\Phi}(L_{S,>}, \rat(K_{S,>}))\xrightarrow{\simeq} \Hom_{D^b_{c}(\Q_S)^\Phi}(L_{S}, \rat(K_{S,>})).
\end{equation*}
Therefore, we obtain a $\Phi$-equivariant morphism $L_{S,>}\to \rat(K_{S,>})$ from the isomorphism $i_S^*j_{d*}(\alpha_{d-1})$, which is unique up to a contractible choice. It is an isomorphism since the triangles split. Since $i_{d*}$ is fully faithful, and $X_d$ consists of disjoint union of these component boundary strata $S$, they glue uniquely to an isomorphism 
\begin{equation*}
    i_{d*}i_d^!L_d[1]\to i_{d*}\rat(w_{>\dim(X_d)+w}i_d^*j_{d*}K_{d-1}).
\end{equation*}
We then define $\alpha_d:L_d\to \rat(K_d)$ to be the induced isomorphism on fibers.

\

$\bullet$ \textbf{Step 3: Hecke equivariance and compatibility.} 

We establish Hecke equivariance in steps. 

\begin{Lemma}\label{Translation I}
\begin{itemize}
    \item [(i.)] Let $g\in G(\Q)$, and $\Gamma':=g^{-1}\Gamma g$. For the isomorphism $c:=c_2:X_{\leq d}^{'BB}\xrightarrow{\simeq} \X_{\leq d}$, there is a commutative diagram
    \begin{equation*}
        \begin{tikzcd}
c^*L_d \arrow[r] \arrow[d, "g^*(\alpha_d)"] & L_d' \arrow[d, "\alpha_d'"] \\
c^*\rat(K_d) \arrow[r]                      & \rat(K_d')      ;                 
\end{tikzcd}
    \end{equation*}
\item[(ii.)] Let $\Gamma'\subseteq \Gamma$ of finite index. Then for the covering map $c:=c_2:X_{\leq d}^{'BB}\to \X_{\leq d}$, there is a commutative diagram
\begin{equation*}
    \begin{tikzcd}
c^*L_d \arrow[r] \arrow[d, "c^*(\alpha_d)"] & L_d' \arrow[d, "\alpha_d'"] \\
c^*\rat(K_d) \arrow[r]                      & \rat(K_d')         .              
\end{tikzcd}
\end{equation*}
\end{itemize}
\end{Lemma}

\begin{proof}
\begin{itemize}
    \item [(i.)] By induction hypothesis (ii.)$_{d-1}$, there is a commutative diagram 
\begin{equation*}
    \begin{tikzcd}
c^*j_{d*}L_{d-1} \arrow[r] \arrow[d, "c^*j_{d*}(\alpha_d)"] & j'_{d*}L'_{d-1} \arrow[d, "j_{d*}(\alpha_d')"] \\
\rat(c^*j_{d*}K_{d-1}) \arrow[r]                            & \rat(j_{d*}K_d')            .                  
\end{tikzcd}
\end{equation*}
We claim first that there is a commutative diagram
\begin{equation*}
    \begin{tikzcd}
{c^*i_{d*}i_d^!L_d[1]} \arrow[r] \arrow[d, "\cong"]          & {i'_{d*}i_{d}^{'!}L'_d[1]} \arrow[d, "\cong"]       \\
\rat(c^*i_{d*}w_{>\dim (X_d)+w}i_d^*j_{d*}K_{d-1}) \arrow[r] & \rat(i'_{d*}w_{>\dim X_d+w}i_d^{'*}j'_{d*}K'_{d-1}).
\end{tikzcd}
\end{equation*}
Since $i_{d*}$ and $i'_{d*}$ are fully faithful, and $X_d$ (resp. $X'_d$) is the disjoint union of component boundary stratum $S$ (resp. $S'$), it suffices to establish the commutative diagram after taking $i_S^*$ (resp. $S'$) for each $S$ (resp. $S'$). Indeed, in this case there is a correspondence between $S$ and $S'$ by the action of $g$. To be more precise, the corresponding maximal parabolic subgroups is $Q'=g^{-1}Qg$, and we have $\Gamma'_Q=g^{-1}\Gamma_Q g$. Fix a divisible $a\in A(\Q)$, under the conjugation $a':=g^{-1}ag\in A'(\Q)$ is divisible. It follows that the weights are preserved under conjugation, namely
\begin{equation*}
    c^*(L_S)_\lambda\subseteq (L'_S)_{\lambda}, c^*(K_S)_{\lambda}\subseteq (K_S)_\lambda. 
\end{equation*}
We then obtain the claimed commutative diagram by \Cref{splitting is defined over Q} and \Cref{Hodge-Looijenga decomposition}. Now we claim there is a "cubic morphism" from the first commutative diagram to the second commutative diagram, as an object in 
\begin{equation*}
    Fun((\Delta^1)^3,D^b_c(X^{'BB}_{\leq d}))\cong Fun(\Delta^1,Fun((\Delta^1)^2, D^b_c(X^{'BB}_{\leq d})).
\end{equation*}
Indeed, by adjunction it suffices to construct the morphism after taking $i_S^*$ (resp. $i_S^{'*}$). Then we obtain the obvious morphism in the category of $\Phi$-equivariant sheaves. The fiber of this cubic morphism gives the desired commutative diagram;

    \item [(ii.)] The proof is identical to (i.), with the following observation: From the proof of \cite[(3.6)]{looijenga19882}, for each $Q$, we can find a common divisible element $a\in A(\Q)$ for $\Gamma$ and $\Gamma'$: 
    \begin{equation*}
        a\Gamma_Q a^{-1}\subseteq \Gamma_Q, a\Gamma'_Qa^{-1}\subseteq \Gamma'_Q.
    \end{equation*}
\end{itemize}    
\end{proof}

Combining the two statements, we obtain the translation compatibility (ii.)$_d$.
 \begin{Corollary}\label{Translation II}
     Let $g\in G(\Q)$ and $\Gamma'\subseteq \Gamma\cap g^{-1}\Gamma g$ of finite index. For the induced map $c=c_2: X_{\leq d}^{'BB}\to \X_{\leq d}$, there is a commutative diagram
      \begin{equation*}
        \begin{tikzcd}
c^*L_d \arrow[r] \arrow[d, "g^*(\alpha_d)"] & L_d' \arrow[d, "\alpha_d'"] \\
c^*\rat(K_d) \arrow[r]                      & \rat(K_d')      .                
\end{tikzcd}
    \end{equation*}
 \end{Corollary}

Now we establish the Hecke equivariance (iii.)$_d$.
\begin{Proposition}\label{Hecke equivariance}
$\alpha_d$ is Hecke equivariant for any $g\in G(\Q)$ and $\Gamma'\subseteq \Gamma\cap g^{-1}\Gamma g$ of finite index.
\end{Proposition}

\begin{proof}
    By \Cref{Translation II}, it remains to show the compatibility under the trace map via $c:=c_1$, and we can assume that $g=1$ by the same corollary, thus $c=c_2$. Moreover, we can reduce to the case where $\Gamma'$ is a normal subgroup of $\Gamma$ by Selberg's lemma. In this case, $X'\to X$ is a finite covering with Galois group $\Delta:=\Gamma/\Gamma'$. By (ii.) of \Cref{Translation I}, the adjunction maps $L_d\to c_*L'_d$ and $\rat(K_d)\to c_*\rat(K'_d)$ are compatible with the isomorphism $\alpha_d$ (with $c_*(\alpha'_d)$). The targets of these two maps are acted by $\Delta$, and the sources are identified with $\Delta$-invariants, to be more precise, the image of the projector
    \begin{equation*}
        |\Delta|^{-1}\sum_{\delta\in \Delta}\delta.
    \end{equation*}
    To see this, one applies induction on $d$, and it is a standard fact in deck transformation to see this for $d=0$. Assume it holds for $d-1$. Since $c$ is a finite map, it preserves Hodge truncations, thus the statement for $K$; the Lie weights are clearly preserved under $c$, hence also the statement for $L$. 
    
    Finally, the trace maps $c_*L'_d\to L_d$ (resp. $c_*K'_d\to K_d$) is a multiple by $|\Delta|$ of the projection of $c_*L'_d$ (resp. $c_*K'_d$) to its direct factor $L_d$ (resp. $K_d$). 
\end{proof}

Finally, we establish the compatibility (iv)$_d$.
\begin{Proposition}
There is a commutative diagram
\begin{equation*}
\begin{tikzcd}
L_d \arrow[r] \arrow[d, "\alpha_d"] & (\pi_*\W^{-\rho+\varepsilon}C(\E))|_{\X_{\leq d}} \arrow[d, "\beta_{d}"] \\
\rat(K_d) \arrow[r]                 & {IC(\E)|_{\X_{\leq d}}    .}                                            
\end{tikzcd}
\end{equation*}   
\end{Proposition}

\begin{proof}
    We use Morel's formula (\Cref{Morel Intersection})
    \begin{equation*}
        IC_X(\E)=w_{\leq d_X+w}j_*\E.
    \end{equation*}
The proof is then identical to (i.) of \Cref{Translation I}.
\end{proof}

Taking $d=r$, we complete the proof of \Cref{Theorem: Lie-Hodge correspondence}. 

\begin{Remark}\label{Flaw of Nair's proof}
    There are several gaps in Nair's proof of \cite[Theorem 4.3.1, Proposition 4.4.1]{nair2012mixed}. Consider two exact triangles in a triangulated category $\CC$, 
    \begin{equation*}
        \begin{tikzcd}
x \arrow[r] \arrow[d, dashed] & y \arrow[r] \arrow[d] & z  \\
x' \arrow[r]                  & y' \arrow[r]          & z',
\end{tikzcd}
    \end{equation*}
with a morphism $y\to y'$ fixed. Then there exists at most one map $x\to x'$ making the diagram commutes if and only if 
\begin{equation*}
    \im(\Hom_\CC(x, z'[-1])\to \Hom_\CC(x, x'))=0.
\end{equation*}
    We use the diagram above to abbreviate the diagram \ref{Construct isomorphism}. In \cite[p.30]{nair2012mixed}, Nair claims that $\alpha_d$ is the unique map making the diagram 
    \begin{equation*}
        \begin{tikzcd}
x \arrow[d, "\alpha_d"] \arrow[r] & y \arrow[d, "j_{d*}(\alpha_{d-1})"] \\
x' \arrow[r]                 & y'                            
\end{tikzcd}
    \end{equation*}
    commutes. The difference of two such lifts is measured by a morphism 
    \begin{equation*}
        \beta: x\to z'[-1].
    \end{equation*}
   By applying $i_S^*$, the lower exact triangle splits, hence $i_S^*(z'[-1]\to x')=0$. Therefore, the inference on p.30 that 
\begin{equation*}
    \mathrm{restricted\ lift\ unique}\implies i_S^*\beta=0
\end{equation*}
is invalid, since it holds for any morphism $\beta$. For the same reason, we can neither deduce that $i_S^*\beta$ is $\Phi$-equivariant nor that the image $\im(\beta)$ vanishes in $\Hom(x, x')$, since both sheaves are not supported on $X_d$ anymore, and the restriction to open and closed strata in general is not faithful. The similar inferences
\begin{equation*}
    \mathrm{restricted\ triangle\ splits}\implies i_S^*(\tilde{\delta})=0
\end{equation*}
on p.31, and $i_S^*(\delta)=0$ on p.33 in the proof of Proposition 4.4.1 of loc. cit, are invalid by the same reasoning.
\end{Remark}

\begin{Remark}\label{why infinity cats}
    Our proof does not work in the setting of triangulated categories. The issue arises from the non-functoriality of cones. For example, in the proof of (i.) of \Cref{Translation I}, we construct the commutative diagram as the fiber of a "cubic morphism". This would fail if one ignore the higher coherence and simply apply (TR3) in triangulated categories. For example, in the derived category of $\C$-vector spaces, consider one square consisting of $\C$ with identities, and the other square $\C[1]$ with identities, with cubic morphism $0$. There are infinitely many non-commutative diagrams on the cocone $\C\oplus \C$ making all other faces commutative. 
\end{Remark}

\section{Weight Comparison and Intersection Cohomology}\label{section: Weight Comparison and Intersection Cohomology}

In this section, we state our main theorem and first consequences on the intersection cohomology of the Baily--Borel compactification.
\subsection{Weight Comparison}
Let $X=\Gamma\backslash G(\R)/KA_G$ be a Hermitian locally symmetric space, and $E$ an irreducible representation of $G$ defined over $\Q$, with the polarizable admissible variation of $\Q$-Hodge structures $\E$ on $X$ of weight $w\in \Z$. Let $\hX$ be the reductive Borel--Serre compactification, and $\X$ the Baily--Borel compactification.

\begin{Theorem}\label{Comparison to RBS}
There is a canonical isomorphism of pure Hodge structures of weight $k+w$
\begin{equation*}
    Gr^W_{k+w}\W^0H^{k}(X,\E)\cong IH^{k}(\X,\E).
\end{equation*}
\end{Theorem}
\begin{proof}
By \Cref{Theorem: Lie-Hodge correspondence} and \Cref{weightless and intersection}, the map 
    \begin{equation*}
        \W^0H^k(X,\E)\to IH^k(\X;\E)
    \end{equation*}
    induced by the map of sheaves
    \begin{equation*}
        \pi_*\W^0C(\E)\to \pi_*\W^{-\rho+\varepsilon}C(\E)\cong IC_{\X}(\E)
    \end{equation*}
    is a morphism of mixed Hodge structures, and the image is identified with $Gr^W_{k+w}\W^0H^k(X,\E)$. It suffices to show that after base change to $\C$, the map 
    \begin{equation*}
        \W^0H^k(X,\E\otimes_\Q\C)\to IH^k(\X;\E\otimes_\Q\C)
    \end{equation*}
is surjective. By Zucker's conjecture, there are natural isomorphisms
    \begin{equation*}
        \pi_*\W^{-\rho+\varepsilon}C(\E\otimes_\Q\C)  \cong IC_{\X}(\E\otimes_\Q\C)\cong A_{(2)}(\E).
    \end{equation*}
    Then the surjectivity follows by \Cref{p to 2 surjectivity}.
\end{proof}

\begin{Remark}\label{Decomposition is canonical}
    Let $X'\twoheadrightarrow \X$ be any resolution of singularities. It follows from \cite[Proposition 4.3.3]{nair2015weightless} that there is a canonical map $H^k(\hX;\Q)\to H^k(X';\Q)$ whose image is exactly $Gr^W_kH^k(\hX;\Q)$. In particular, $IH^k(\X;\Q)$ is canonically a subspace of $H^k(X';\Q)$. For general varieties, such an identification by the decomposition theorem is not canonical.
\end{Remark}

\begin{Remark}\label{algebraic proof?}
    We sketch how a potential local and algebraic proof might suggest. The general theory of weight truncations implies that the complex $w_{>\underline{\dim}}IC_X(\Q)$ has weights $\leq 0$. Ideally, one could conclude the proof if one shows further that it has weights $\leq -1$. For simplicity assume that $G^{ad}$ is $\Q$-simple and has $\Q$-rank $1$. In this case, the boundary strata of $\X$ consist of isolated compact Hermitian locally symmetric varieties of smaller dimensions. Let $i:S\hookrightarrow \X$ be such a stratum. Using the theory of local Hecke operators, the weight filtration on the local intersection cohomology $H^ki^*IC_X(\Q)$ is split over $\Q$ (\cite[(6.4)]{looijenga1991weights}), and $H^ki^*w_{\leq \underline{\dim}}IC_X(\Q)$ consists of the pieces of weights $\leq 0$. The condition that the complex $w_{>\underline{\dim}}IC_X(\Q)$ has weights $\leq -1$ then implies that $H^ki^*IC_X(\Q)$ has weights $\leq k-1$ when $k\neq 0$.
\end{Remark}

\subsection{Operations on Intersection Cohomology}
We restrict to the case for $\E=\Q$ in this subsection, even though the general coefficients follow verbatim. There is an isomorphism $\W^0C(\Q)\cong \Q_{\hX}$.

\begin{Proposition}\label{cup product on IH}
There exists a canonical cup product on $IH^*(\X;\Q)$ compatible with the canonical maps
    \begin{equation*}
        H^k(\X;\Q) \to IH^*(\X;\Q)\to H^*(X;\Q).
    \end{equation*}
Moreover, in the dual dimensions, the cup product
\begin{equation*}
    IH^k(\X;\Q)\times IH^{2n-k}(\X;\Q)\xrightarrow{\cup} IH^{2n}(\X;\Q)\cong \Q(-n)
\end{equation*}
is the perfect pairing induced by the Poincaré-Verdier duality on $IC_X$.    
\end{Proposition}
\begin{proof}
    By \Cref{Properties of EC} and \Cref{Theorem: Lie-Hodge correspondence}, the cup product on $H^*(\hX;\Q)$ is a morphism of mixed Hodge structures. Therefore, it defines a cup product on $IH^*(\X;\Q)$ by \Cref{Comparison to RBS}. Moreover, since the cup products commute with the canonical maps
    \begin{equation*}
         H^*(\X;\Q)\to H^*(\hX;\Q)\to H^*(X;\Q), 
    \end{equation*}
they are also compatible with $IH^*(\X;\Q)$. Finally, we must show that this cup product realizes the Poincaré duality defined via the Verdier duality on $IC_X$. It suffices to show that the cup product on $H^*(\hX;\Q)$ is compatible with the intersection product on $IH^*(\X;\Q)$ along the surjection $H^*(\hX;\Q)\to IH^*(\X;\Q)$. We show that there is a commutative diagram 
\begin{equation*}
\begin{tikzcd}
\Q_{\hX}\otimes\Q_{\hX} \arrow[r, "m_{\hX}"] \arrow[d]               & \Q_{\hX} \arrow[d, "tr_{\hX}"] \\
\W^{-\rho}C(\Q)\otimes \W^{-\rho+\varepsilon}C(\Q) \arrow[r, "\Phi"] & {\omega_{\hX}[-2n]}        .   
\end{tikzcd}
\end{equation*}
Here $m_{\hX}$ is the multiplication map, adjoint to the identity on $\Q_{\hX}$, realizing the cup product. $\omega_{\hX}$ is the dualizing complex on $\hX$ (It is isomorphic to the weighted complex $\W^{-2\rho+\varepsilon}C(\Q)$ by \cite[(19.4)]{goresky1994weighted}), and $n:=\dim(X)$. The right map $tr_{\hX}$ is induced by the trace map $tr_{\hX}:H^{2n}(\hX;\Q)\xrightarrow{\sim} \Q(-n)$. The map $\Phi$ is the Verdier duality on weighted cohomology (\cite[(21.4)]{goresky1994weighted}), and its pushforward onto $\X$ realizes the Verdier duality on $IC_{\X}$, since they are both the unique map extending the identity on $\Q_X$. Consequently, the commutativity of this diagram implies that the cup products we define on $IH^*(\X;\Q)$ realize the Poincaré duality on it in the dual dimensions. 

In order to show the commutativity, we note that there is a canonical isomorphism
\begin{equation*}
    \Hom(\Q_{\hX}\otimes \Q_{\hX}, \omega_{\hX}[-2n])\cong H^{2n}_c(\hX;\Q)^*(\cong H^{2n}(\hX;\Q)^*).
\end{equation*}
Since $H^{2n}_c(X;\Q)\cong H^{2n}(\hX;\Q)\cong \Q(-n)$,  any morphism between these two sheaves is determined by the restriction to $X$. It then suffices to check the commutativity on $X$, where all sheaves reduce to the constant sheaf $\Q_X$. 
\end{proof}

Now we construct the pullback of the intersection cohomology, and show that it satisfies the usual properties of the pullback. Let $H$ be another connected reductive group over $\Q$, such that the associated symmetric space is Hermitian. Let $\Gamma_H\subseteq H(\Q)$ be a neat arithmetic subgroup, and $X_H:=\Gamma_H\backslash H(\R)/K_H$ be the Hermitian locally symmetric variety. It follows from \cite[Theorem 2]{kiernan1972satake} that any morphism $\iota:X_H\to X$ extends uniquely to a morphism $\iota^{BB}: \X_H\to \X$. In particular, any morphism of Shimura data induces such a map of connected Hermitian locally symmetric varieties and extends to the Baily--Borel compactification. On the other hand, reductive Borel--Serre compactification is not functorial. However, its cohomology admits pullbacks, by the identification \Cref{Theorem: Lie-Hodge correspondence} and \Cref{Properties of EC}. 

\begin{Proposition}\label{pullback on IH}
Let $n=\dim(X), m=\dim(X_H), r=n-m$. 
\begin{itemize}
    \item [(i.)] There exists a morphism of graded rings of Hodge structures
\begin{equation*}
     \iota^*:IH^k(\X;\Q)\to IH^k(\X_H;\Q),
\end{equation*}
 commuting with $(\iota^{BB})^*$ on $H^k(\X;\Q)$ and $\iota^*$ on $H^k(X;\Q)$.

 \item[(ii.)]  The Gysin map
   \begin{equation*}
       \iota_*: IH^k(\X_H;\Q)\to IH^{k+2r}(\X;\Q(r)),
   \end{equation*}
   induced by the Poincaré duality of $\iota^*$, satisfies the projection formula
\begin{equation*}
    \iota_*(\alpha\cup \iota^*\beta)=\iota_*\alpha\cup \beta,
\end{equation*}
for $\alpha\in IH^*(\X_H;\Q)$ and $\beta\in IH^*(\X;\Q)$.
\end{itemize}
\end{Proposition}
\begin{proof}
    The existence of $\iota^*$ follows from \Cref{Comparison to RBS} and \Cref{Properties of EC}. The map $\iota^*$ on $IH^k$ commutes with $\iota^*$ on $H^k$ and the cup product, since so does
    \begin{equation*}
        \iota^*:H^k(\hX;\Q)\to H^k(\hX_H;\Q).
    \end{equation*}

For $\alpha\in IH^k(\X_H;\Q)$ and $\beta\in IH^{2m-k}(\X;\Q)$, the projection formula follows by the definition of the Gysin map and Poincaré duality (\Cref{cup product on IH}).
In general, it follows by the fact that $\iota^*$ commutes with cup products. 
\end{proof}

Now we introduce the notion of modular intersection class. Let $H\subseteq G$ be a subgroup, and $\iota:X_H\to X$ be a Hermitian locally symmetric subvariety. 
\begin{Definition}\label{modular intersection class}
    Let $[X_H]\in H^0(X_H;\Q)\cong IH^0(\X_H;\Q)$ be the fundamental class of $X_H$. Then the \textit{modular intersection class} of $X_H$ is defined as
    \begin{equation*}
        \Icl (X_H):=\iota_*[X_H]\in IH^{2r}(\X;\Q(r)).
    \end{equation*}
\end{Definition}

\begin{Proposition}\label{properties of Intersection cycle}
        The image of the modular intersection class $\Icl(X_H)$ under the canonical map
    \begin{equation*}
        IH^{2r}(\X;\Q(r))\twoheadrightarrow W_{2r}H^{2r}(X;\Q(r))\hookrightarrow H^{2r}(X;\Q(r))
    \end{equation*}
    is the usual cycle class of $X_H$ in $X$. Let 
    \begin{equation*}
        \deg_X:IH^{2n}(\X;\Q(n))\cong H^{2n}_c(X;\Q(n))\to \Q
    \end{equation*}
    be the degree map (or trace map), then for any $\beta\in IH^{2m}(\X;\Q(m))$,
    \begin{equation*}
        \deg_X(\Icl(X_H)\cup \beta)=\deg_{X_H}(\iota^*\beta).    \end{equation*}
\end{Proposition}
\begin{proof}
    The first statement holds since the Gysin map on $IH^*$ commutes with the usual Gysin map on $H^*$. The second statement follows by the projection formula
    \begin{equation*}
        \Icl(X_H)\cup \beta= \iota_*\iota^*\beta.
    \end{equation*}
\end{proof}

\begin{Remark}\label{modular intersection class is universal}
    In view of \Cref{Decomposition is canonical}, the modular intersection class is a universal cohomology class for modular subvarieties in all resolutions of $\X$, in particular, in any smooth toroidal compactification. On the other hand, the map from modular cycles to $IH^{2r}(\X;\Q(r))$ does not factor through the Chow group $CH^r(X)_\Q$ of the open variety. For example, when $X\cong\A^1_\C-\{0,1\}$ is a smooth open modular curve, $CH^1(X)_{\Q}=0$; while the modular intersection class of a special point in $IH^2(\X;\Q(1))\cong \Q$ is nonzero.
\end{Remark}

Our construction of the pullback agrees with that in \cite[Chapter 9]{bergeron2006proprietes} using $L^2$-harmonic forms as follows. In \cite[Lemma 3.10]{bergeron2006proprietes}, it is shown that every $L^2$-harmonic form on $X$ is bounded, and therefore it restricts to a closed smooth bounded form on $X_H$ (not harmonic in general), then one takes its $L^2$-cohomology class. Let $j:X\hookrightarrow \X$. By \Cref{Properties of EC}, there exists a commutative diagram of sheaves on $\X$
\begin{equation*}
    \begin{tikzcd}
\pi_*\C_{\hX} \arrow[r] \arrow[d] & \iota^{BB}_*\pi_{H*}\C_{\hX_H} \arrow[d] \\
j_{*}\C_{X} \arrow[r]             & \iota^{BB}_*j_{H*}\C_{X_H}      ,        
\end{tikzcd}
\end{equation*}
The bottom arrow commutes with the restriction of the smooth differential forms on the open varieties by the functoriality of the de Rham isomorphism of smooth manifolds. Since our construction is a descent from the pullback on the cohomology of the reductive Borel--Serre compactification, it agrees with the analytic construction, which is essentially the restriction of differential forms. 

With this observation, we can show that the analytic geometric volumes agree with the topological degrees of the modular intersection classes.
\begin{Proposition}\label{geometric volume=degree}
     Let $\iota: X_H\rightarrow X$ be a modular subvariety of dimension $m$. Then for any smooth bounded closed $2m$-form $\alpha$ on $X$, which defines a class $[\alpha]$ in the $L^2$-cohomology/intersection cohomology, we have
    \begin{equation*}
        \mathrm{vol}(X_H,\alpha):=\int_{X_H} \iota^*\alpha=\deg_{X}(\Icl(X_H)\cup [\alpha]).
    \end{equation*}
\end{Proposition}
\begin{proof}
    By \Cref{properties of Intersection cycle}, we have
    \begin{equation*}
        \deg_{X}(\Icl(X_H)\cup [\alpha])=\deg_{X_H}(\iota^*[\alpha]).
    \end{equation*}
    By definition, 
    \begin{equation*}
        \deg_{X_H}(\iota^*[\alpha])=\int_{X_H} \omega,
    \end{equation*}
    where $\omega$ is any compactly supported differential form on $X_H$ representing $\iota^*[\alpha]$, under the isomorphism
    \begin{equation*}
        IH^{2m}(\X_H;\C)\cong H^{2m}_c(X_H;\C).
    \end{equation*}
The $L^2$-cohomology is isomorphic to the compactly supported cohomology on the top degree via integration along the fundamental class by $L^2$-Stokes formula
    \begin{equation*}
    \begin{aligned}
         H^{2m}_c(X_H;\C)&\cong H^{2m}_{(2)}(X_H)\cong \C.
    \end{aligned}
    \end{equation*}
    Since $\iota^*[\alpha] =[\iota^*\alpha]$ by the above observation, the image of $[\omega]$ is $[\iota^*\alpha]$ under the above isomorphism, and we have 
    \begin{equation*}
\deg_{X_H}(\iota^*[\alpha])=\int_{X_H}\omega=\int_{X_H}\iota^*\alpha.
    \end{equation*}
\end{proof}

\section{Further Applications}\label{Section: Applications}
In this section, we present two applications of our results. We adopt the convention and notation from the original papers, and therefore may differ from the previous sections. In \Cref{subsection: Auto Lef}, we follow the convention in \cite{nair2023automorphic}; while in \Cref{subsection: Coh Siegel-Weil}, we follow the convention in \cite{kudla49special}, and this is the only part we adopt the adelic language.
\subsection{Automorphic Lefschetz Properties for Intersection Cohomology}\label{subsection: Auto Lef}
In this subsection, we extend the results in \cite{venkataramana2001cohomology} to the intersection cohomology of (noncompact) Shimura varieties. We will only consider cohomology with coefficients in $\C$ as in \cite{venkataramana2001cohomology} and \cite{nair2023automorphic}, we will simply abbreviate $H^k(X;\C)$ by $H^k(X)$ for an algebraic variety $X$. All maps on cohomology groups will be morphisms of (filtered colimits of) mixed Hodge structures up to proper Tate twists, and therefore we will ignore them in this subsection.

We summarize the construction of the restriction map, following \cite[\S 2, \S 3]{nair2023automorphic}. Let $G$ be a connected semisimple algebraic group over $\mathbb{Q}$, almost simple modulo its center, with $G(\mathbb{R})$ semisimple and noncompact modulo a compact center. Let $\Gamma\subseteq G(\Q)$ be a neat congruence subgroup. Assume that the connected locally symmetric space $X_\Gamma=\Gamma\backslash G(\R)/K$ is Hermitian. For $\Gamma'\subseteq \Gamma$, the covering map $X_{\Gamma'}\to X_{\Gamma}$ induces via pullback maps on cohomology
\begin{equation*}
    H_{(c)}^*(X_{\Gamma})\to H_{(c)}^*(X_{\Gamma'}),\ \ \ IH^*(\X_{\Gamma})\to IH^*(\X_{\Gamma'}).
\end{equation*}
These maps are all injective. Taking colimits over all congruence subgroups, we can define
\begin{equation*}
    H^k_{(c)}(X):=\colim_\Gamma  H^k_{(c)}(X_\Gamma),\ \ \ IH^k(\X):=\colim_\Gamma IH^k(\X_{\Gamma}).
\end{equation*}

All of the above are natural smooth left $G(\Q)$-modules, in the sense that the stabilizer of an element is a congruence subgroup. The action can be explicitly defined as follows. For any $g\in G(\Q)$ and $\Gamma\subseteq G(\Q)$, there is an isomorphism
\begin{equation*}
\begin{aligned}
     g^{-1}: X_{g\Gamma g^{-1}}&\xrightarrow{\sim} X_\Gamma\\
     g\Gamma g^{-1}x&\mapsto \Gamma(g^{-1}x).
\end{aligned}
\end{equation*}
Then the action of $g$ on $H^k_{(c)}(X)$ and $IH^k(\X)$ is defined as the pullback $(g^{-1})^*$ followed by passing to the colimit. In particular, for a (normal) congruence subgroup $\Gamma'\subseteq \Gamma$ of finite index, with $\Gamma'\subseteq \Gamma\cap g^{-1}\Gamma g$, we can always regard $g^*\alpha$ as an element in $H^k_{(c)}(X_{\Gamma'})$ (resp. $IH^k(\X_{\Gamma'})$) for $\alpha\in H^k_{(c)}(X_\Gamma)$ (resp. $\alpha\in IH^k(\X_{\Gamma})$). Note that the $G(\Q)$-module $IH^*(\X)$ is semisimple by the Matsushima formula, see \cite[Proposition 3.6]{nair2023automorphic} or \cite[Proposition 1.5 (1)]{venkataramana2001cohomology}.

\begin{Lemma}\label{Hecke preserves cup product}
    Any $g\in G(\Q)$ acts as a ring homomorphism on $IH^*(\X)$.
\end{Lemma}
\begin{proof}
    The morphism $g^{-1}$ extends uniquely to an isomorphism 
    $\hX_{g\Gamma g^{-1}}\xrightarrow{\sim}\hX_\Gamma$, and therefore $g$ acts on $H^*(\hX)$ by the pullback $(g^{-1})^*$, which preserves cup product. Then by \Cref{cup product on IH}, it also preserves the cup product on $IH^*(\X)$. 
\end{proof}

Let $X^c$ be the compact dual symmetric space of $G(\R)/K$, defined as follows. Let $\g_0$ be the real Lie algebra of $G(\R)$, and the Cartan decomposition
\begin{equation*}
    \g_0=\mathfrak{k}_0\oplus \mathfrak{p}_0\subseteq \g.
\end{equation*}
Then the compact form
\begin{equation*}
    \mathfrak{u}_0:=\mathfrak{k}_0\oplus i\mathfrak{p}_0\subseteq \g
\end{equation*}
induces a compact group $U$ containing $K$. Then $X^c:=U/K$. It is known that there is an embedding by the Matsushima formula
\begin{equation*}
    H^k(X^c)\hookrightarrow IH^k(\X),
\end{equation*}
identifying $H^k(X^c)$ with $G(\Q)$-invariants of $IH^k(\X)$, see \cite[Proposition 3.6]{nair2023automorphic} or \cite[Proposition 1.5 (3)]{venkataramana2001cohomology}. We note that the fundamental class $[X^c]\in H^{2n}(X^c)$ defines a canonical generator of $IH^{2n}(\X_\Gamma)$ that is compatible when $\Gamma$ varies. Therefore, the top-degree intersection cohomology is canonically isomorphic to the trivial $G(\Q)$-module
\begin{equation*}
    H^{2n}(X^c)\cong IH^{2n}(\X)\cong \C.
\end{equation*}

Now we define the restriction map. Let $H\subseteq G$ be a connected reductive subgroup over $\mathbb{Q}$. We choose some maximal compact subgroup $K\subseteq G(\mathbb{R})$, such that the group $H(\mathbb{R})\cap K$ is a maximal compact subgroup of $H(\mathbb{R})$. Assume that $H(\R)/K_H$ is also Hermitian, and the map
\begin{equation*}
    \iota:X_{H,\Gamma_H}\hookrightarrow X_{\Gamma}
\end{equation*}
 is a closed immersion of quasi-projective varieties. Here $\Gamma_H:=\Gamma\cap H(\Q)$.
 
By \Cref{pullback on IH}, there are canonical maps 
\begin{equation*}
    \iota^*:H^*_{(c)}(X_{\Gamma})\rightarrow H^*_{(c)}(X_{H,\Gamma_H}), \ \ IH^*(\X_{\Gamma})\rightarrow IH^*(\X_{H,\Gamma_H}),
\end{equation*}
which are compatible with the natural maps $H^*_c\to IH^*\to H^*$. The restriction map is then defined as 
\begin{equation*}
    Res:=\prod_{g\in G(\Q)}\iota^*g^*: H^k(X)\to \prod_{g\in G(\Q)} H^k(X_H),
\end{equation*}
and similar for $H^k_c$ and $IH^k$. There are equivalent definitions in \cite{venkataramana2001cohomology} using the adelic language and in \cite{nair2023automorphic} using the induced modules.

 Let $n$ (resp. $m$) denote the (complex) dimension of $X_\Gamma$ (resp. $X_{H,\Gamma_H}$), and $r=n-m$. We denote by
 \begin{equation*}
     \xi_\Gamma:=\Icl(X_{H,\Gamma_H})=\iota_*([X_{H,\Gamma_H}])\in IH^{2r}(\X_\Gamma)\subseteq IH^{2r}(\X)
 \end{equation*}
 the modular intersection class (\Cref{modular intersection class}), and define
 \begin{equation*}
     \xi^c:= \iota^c_*([X^c_H])\in H^{2r}(X^c)\subseteq IH^{2r}(\X).
 \end{equation*}

The following results extend those of \cite{venkataramana2001cohomology} as well as \cite{nair2023automorphic}. The proofs are verbatim, and we sketch them here for convenience. 

\begin{Lemma}\label{Venkataramana 1}
    Let $V_\Gamma$ be the $G(\Q)$-submodule of $IH^{2r}(\X)$ generated by $\xi_\Gamma$. Then the space of $G(\Q)$-invariants in $V_{\Gamma}$ is generated by $\xi^c$.
\end{Lemma}
This corresponds to \cite[Theorem 1]{venkataramana2001cohomology} and \cite[Proposition 3.10]{nair2023automorphic}.
\begin{proof}
    The space $V^0:=V_{\Gamma}^{G(\Q)}\subseteq H^{2r}(X^c)$ is at most one-dimensional, since $V_\Gamma$ is cyclic. By semisimplicity, $V=V^0\oplus V^1$ where $V^1$ is a $G(\Q)$-submodule complement to $V^0$. We write 
    \begin{equation*}
        \xi_\Gamma=\xi_\Gamma^0+\xi_{\Gamma}^1, \ \ \xi^0_{\Gamma}\in V^0, \ \ \xi^1_\Gamma\in V^1.
    \end{equation*}
 We claim that $\xi^0_\Gamma=\xi^c$, which is sufficient for the proof since $V^0$ is at most one-dimensional and it implies $0\neq \xi^c\in V^0$.
    
As $G(\Q)$ acts on $IH^*(\X)$ as ring homomorphisms by \Cref{Hecke preserves cup product},  and $IH^{2n}(\X)\cong H^{2n}(X^c)$ is a one-dimensional trivial $G(\Q)$-module, for any $\alpha\in H^{2m}(X^c)$, we have $\alpha\cup V^1_\Gamma=0$. Indeed, for any $u\in V^1_\Gamma$, and any $g\in G(\Q)$, we have
\begin{equation*}
    \alpha\cup u =g(\alpha\cup u)= g(\alpha)\cup g(u)=\alpha\cup g(u).
\end{equation*}
This implies that 
\begin{equation*}
    \alpha\cup-: V^1_\Gamma\to IH^{2n}(\X)\cong \C
\end{equation*}
is $G(\Q)$-equivariant. But $V^1_\Gamma$ has no $G(\Q)$-invariants by definition,  hence no $G(\Q)$-coinvariants by semisimplicity. It follows that
\begin{equation*}
    \alpha\cup \xi_\Gamma = \alpha\cup \xi^0_\Gamma
\end{equation*}
for any $\alpha\in H^{2m}(X^c)$. On the other hand, since  
\begin{equation*}
    \begin{tikzcd}
H^0(X^c_H) \arrow[r, "\sim"] \arrow[d, "\iota^c_*"'] & IH^0(\X_H) \arrow[d, "\iota_*"] \\
H^{2r}(X^c) \arrow[r, hook]                        & IH^{2r}(\X)                  
\end{tikzcd}
\end{equation*}
by Poincaré duality, we have
\begin{equation*}
    \alpha\cup \xi_\Gamma = \alpha\cup \xi^c,\ \ \ \alpha\in H^{2m}(X^c).
\end{equation*}
It follows that $\xi^0_\Gamma=\xi^c\in H^{2r}(X^c)$ by Poincaré duality.

\end{proof}

\begin{Theorem}\label{Venkataramana's Criterion}
     If $\alpha\in IH^k(\X)$ and $Res(\alpha)=0$, then $\alpha\cup \xi^c=0$.
\end{Theorem}
This corresponds to \cite[Theorem 2]{venkataramana2001cohomology} and \cite[Theorem 3.11]{nair2023automorphic}.
\begin{proof}
    Let $\alpha\in IH^k(\X_{\Gamma})$ for some congruence subgroup $\Gamma\subseteq G(\Q)$, and $g\in G(\Q)$. Then we can find a normal subgroup $\Gamma'$ of finite index in $\Gamma$, with $\Gamma'\subseteq \Gamma\cap g^{-1}\Gamma g$. Let $\iota':X_{H,\Gamma_H'}\hookrightarrow X_{\Gamma'}$ be the embedding. Then for any $\gamma\in \Gamma/\Gamma'$, 
    \begin{equation*}
        (\iota')^*(g\gamma)^*\alpha=(\iota')^*\gamma^*g^*\alpha=0\in IH^{k}(\X_{H,\Gamma'})
    \end{equation*}
      It follows by the projection formula (\Cref{pullback on IH}) that
    \begin{equation*}
            0=\iota'_*(\iota')^*\gamma^*g^*\alpha
            =\gamma^*g^*\alpha\cup \xi_{\Gamma'}
            =g^*\alpha\cup (\gamma^{-1})^*\xi_{\Gamma'}.  
    \end{equation*}
By \cite[Lemma 4.2]{venkataramana2001cohomology}, we have
    \begin{equation*}
        \xi_\Gamma=|\Gamma/\Gamma'|\sum_{\gamma\in \Gamma/\Gamma'} \gamma^* \xi_{\Gamma'}\in IH^{2(n-m)}(\X).
    \end{equation*}
It follows that
\begin{equation*}
   0= \sum_{\gamma\in \Gamma/\Gamma'}g^*\alpha\cup \gamma^*\xi_{\Gamma'}= |\Gamma/\Gamma'|^{-1}g^*\alpha\cup \xi_{\Gamma},
\end{equation*}
    and therefore
    \begin{equation*}
        g^*\alpha\cup \xi_{\Gamma}=\alpha\cup (g^{-1})^*\xi_{\Gamma}=0,
    \end{equation*}
for any $g\in G(\Q)$. Now summing over $G(\Q)/\Gamma$, one concludes that
\begin{equation*}
    \alpha\cup \xi^c=0,
\end{equation*}
since $\xi^c$ is contained in the $G(\Q)$-module generated by $\xi_\Gamma$, by \Cref{Venkataramana 1}.
\end{proof}
This implies all the consequences \cite[Theorem 3-8]{venkataramana2001cohomology} in the noncompact cases. We will not list all of them, but only record the necessary results relevant to the questions in \cite{nair2023automorphic} in the unitary case.
\begin{Corollary}\label{AutoLef}
    If $n=m+1$, then
    \begin{equation*}
        Res:IH^k(\X)\to \prod_{g\in G(\Q)} IH^k(\X_H)
    \end{equation*}
    is injective for $k\leq m$.
\end{Corollary}
This corresponds to \cite[Theorem 3]{venkataramana2001cohomology}.
\begin{proof}
    This follows by \Cref{Venkataramana's Criterion} and hard Lefschetz, since $\cup\xi^c$ defines a canonical Lefschetz operator on $IH^*(\X)$ (or $H^*_{(2)}(X)$) in this case. See \cite[\S 5]{venkataramana2001cohomology} for more detail. 
\end{proof}

The following corollary extends the result in \cite{nair2023automorphic} in the unitary case. In particular, we answer the question in \cite[\S 1.4]{nair2023automorphic} for the unitary Shimura varieties in the middle degree. Let $G(\R)^{nc}$ denote the product of the noncompact factors of $G(\R)$.

\begin{Corollary}\label{Auto Lef Remark}
If $H(\R)^{nc}\subseteq G(\R)^{nc}$ is $SU(1,m)\subseteq SU(1,n)$, then
       \begin{equation*}
           Res:H^k(X)\to \prod_{g\in G(\Q)} H^k(X_H)
       \end{equation*}
is injective for $k\leq m$. When $k=m$, the image is necessarily contained in the product of $H^m_!(X_H)$.
\end{Corollary}
This is known in \cite{nair2023automorphic} when $k<m$ and $H\subseteq G$ of the same $\Q$-rank.
\begin{proof}
By induction, one can reduce to the case $m=n-1$. Then \Cref{AutoLef} applies, and we get
\begin{equation*}
    Res: IH^k(\X)\to \prod_{g\in G(\Q)} IH^k(\X_H)
\end{equation*}
is injective for $k\leq n-1$. We note that the boundary of $\X$ consists of isolated points, thus
\begin{equation*}
    IH^k(\X)\cong H^k(X)
\end{equation*}
    for $k\leq m=n-1$, and $IH^m(\X_H)\cong H^m_{!}(X_H)$, see \cite[Proposition 6.7.3]{maxim2019intersection}. This gives the injection for $k\leq m$.
\end{proof}

\subsection{A Cohomological Siegel--Weil Formula}\label{subsection: Coh Siegel-Weil}

The aim of this final section is to give a cohomological interpretation of Kudla's analytic Siegel--Weil formula, recalled in \Cref{generating series of volumes}. In the compact case, such an interpretation was established in \cite[(4.4)]{kudla49special}, via Kudla--Millson theory.

We summarize the constructions, following \cite[\S 1-4]{kudla49special}. Let $G=\mathrm{GSpin}(V)$, where $(V,(\ ,\ ))$ is an inner product space over $\Q$ of signature $(n,2)$ with Witt index $\mathrm{Witt}(V)$.\footnote{i.e., the dimension of a maximal isotropic subspace.} Let 
\begin{equation*}
    D=\{w\in V(\C)|(w,w)=0, (w,\overline{w})<0\}/\C^{\times}\subseteq \P(V(\C)).
\end{equation*}
This determines a Shimura datum $X=Sh(G,D)$, and for $K\subseteq G(\A^{\infty})$ a compact open subgroup, one has a complex quasi-projective variety of dimension $n$
\begin{equation*}
    X_K(\C)=G(\Q)\backslash (D\times G(\A^\infty)/K),
\end{equation*}
where $\A^{\infty}$ is the ring of finite adèles.

Given an $r$-tuple of vectors $x=(x_1,\ldots,x_r)\in V(\Q)^r$,  assume that the matrix
\begin{equation*}
    Q(x):=\frac{1}{2}((x_i,x_j))
\end{equation*}
is positive semidefinite of rank $r(x)$, then it determines a sub-Shimura variety $Sh(G_x,D_x)\subseteq Sh(G,D)$ of codimension $r(x)\leq r$, where 
\begin{equation*}
    \begin{aligned}
        G_x&:=\mathrm{stab}_G(x)\\
        D_x&:=\{w\in D|(x,w)=0\}/\C^{\times}\subseteq D.
    \end{aligned}
\end{equation*}
For any $g\in G(\A^\infty)$ and $y=gx\in V(\A^\infty)^r$, define the Hecke-translated algebraic cycle
\begin{equation*}
\begin{aligned}
    Z(y;K):G_x(\Q)\backslash(D_x\times G_x(\A^\infty)/K^g_x)&\to G(\Q)\backslash (D\times G(\A^\infty )/K)=X_K(\C),\\
    G_x(\Q)(z,h)K^g_x&\mapsto G(\Q)(z, hg)K,
\end{aligned}  
\end{equation*}
where $K^g_x:=G_x(\A^\infty)\cap gKg^{-1}$. 

Finally, for a $K$-invariant Schwartz function $\varphi\in S(V(\A^\infty)^r)^K$ and a symmetric positive semidefinite matrix $T\in Sym_r(\Q)_{\geq 0}$, define the weighted modular cycle
\begin{equation*}
Z(T,\varphi;K)
  := \sum_{\substack{y \in K\backslash V(\A^\infty)^r\\ Q(y)=T}} \varphi(y)Z(y;K)\in Z^{2r(T)}(X_K(\C)).
\end{equation*}
For $K'\subseteq K$ and let $p:M_{K'}\to M_K$ be the covering map, then 
\begin{equation*}
    p^*Z(T,\varphi;K)=Z(T,\varphi;K').
\end{equation*}

Let $\L$ be the line bundle on $X_K(\C)$, the descent of the tautological bundle $\O(-1)$ on $\P(V(\C))$. Let $\ell:=c_1(\L^{\vee})$ and $\Omega$ be the Chern form of $\L^{\vee}$. Then $-\Omega$ is a canonical invariant Kähler form on $X_K$\footnote{The sign comes from the convention on the orientation of $X_K$.} (\cite[Corollary 4.12]{kudla2003integrals}). Note also that the line bundle $\L^{\vee}$ extends canonically to an ample line bundle on $\X_K$.

 Let 
\begin{equation*}
    \mathbb{H}^r:=\{\tau=u+iv|u,v\in Sym_r(\R), v>0\}
\end{equation*}
be the Siegel upper half plane of degree $r$. For each $r$, there exists a canonical closed $2r$-form $\theta_r(\tau,\varphi)$ on $X_K(\C)$ (\cite[(3.4)]{kudla49special}, see \cite[\S 4]{kudla1988tubes} for more detail). For a modular subvariety $\iota:Z\hookrightarrow X_K$ of codimension $m$ and any smooth bounded closed $2m$-form $\alpha$ on $X_K$, we define
    \begin{equation*}
        \mathrm{vol}(Z,\alpha):=\int_Z\iota^*\alpha,
    \end{equation*}
with respect to the induced metric on $Z$, and this definition extends linearly to modular cycles of $X$. We now recall Kudla's Siegel--Weil formula \cite[Theorem 4.1]{kudla49special}.

\begin{Theorem}\label{generating series of volumes}
    When Weil's condition $r<n+1-\mathrm{Witt}(V)$ holds, the theta integral
    \begin{equation*}
        I_r(\tau,\varphi):=\int_{X_K}\theta_{r}(\tau,\varphi)\wedge \Omega^{n-r}
    \end{equation*}
    converges absolutely, and there are identities
    \begin{equation*}
        \mathrm{vol}(X_K,\Omega^n)E_r(\tau,s_0,\varphi)= I_r(\tau,\varphi)=\sum_{T\geq 0}\mathrm{vol}(Z(T,\varphi;K),\Omega^{n-r(T)})q^T.
    \end{equation*}
    Here $q^T:=e(\mathrm{tr}(T\tau))$, and $E_r(\tau, s_0,\varphi)$ is a Siegel Eisenstein series associated to $\varphi$, see \cite[\S 4]{kudla49special}, and $s_0=(n+1-r)/2$.
\end{Theorem}

We give a cohomological interpretation of this formula. The coefficients can be interpreted as the degrees of the modular intersection classes by \Cref{geometric volume=degree}, and we must show that the generating series of the classes converges.

\begin{Lemma}\label{Volume Dominance}
    For any smooth closed bounded $2m$-form $\alpha$ on $X_K$, and any closed modular subvariety $\iota:Z\hookrightarrow X_K$ of dimension $m$, we have
    \begin{equation*}
     \left|\int_Z \iota^*\alpha \right|\leq \int_Z \left|\iota^*\alpha \right|d\mathrm{vol}_Z\leq \frac{\|\alpha\|_{L^{\infty}(X)}}{m!}\left|\int_Z \iota^*\Omega^m\right|.
    \end{equation*}
\end{Lemma}
\begin{proof}
    For any point $z\in Z$, we have
    \begin{equation*}
        |(\iota^*\alpha)_z|\leq |\alpha_{\iota(z)}|\leq \|\alpha\|_{L^\infty(X)}.
    \end{equation*}
    The first inequality holds since the metric on $Z$ is induced from $X_K$. Thus
    \begin{equation*}
        \int_Z|\iota^*\alpha|d\mathrm{vol}_Z\leq \|\alpha\|_{L^\infty(X)}\left|\int_Z d\mathrm{vol}_Z\right|.
    \end{equation*}
 Since $-\Omega$ is a Kähler form on $X$, $-\iota^*\Omega$ is the induced Kähler form on $Z$, we have
 \begin{equation*}
     (-1)^m\iota^*\Omega^m=m!d\mathrm{vol}_Z
 \end{equation*}
   by Wirtinger's formula. Therefore,
    \begin{equation*}
        \int_Z|\iota^*\alpha|d\mathrm{vol}_Z\leq \frac{\|\alpha\|_{L^{\infty}(X)}}{m!}\left|\int_Z \iota^*\Omega^m\right|.
    \end{equation*}
\end{proof}

\begin{Corollary}\label{Modularity of modular intersection class}
    For $\tau\in \mathbb{H}^r$, the formal sum of the modular intersection classes 
    \begin{equation*}
        \phi_r(\tau,\varphi):=\sum_{T\in Sym_r(\Q)_{\geq0}}\Icl(Z(T,\varphi;K))\cup \ell^{r-r(T)}q^T
    \end{equation*}
 converges absolutely in $IH^{2r}(\X;\C)$ when $r<n+1-\mathrm{Witt}(V)$.
\end{Corollary}
\begin{proof}
By linearity, it suffices to assume that $\varphi\geq0$ is a real-valued non-negative Schwartz function. Then the coefficients of $Z(y;K)$ in $Z(T,\varphi;K)$ are all non-negative. It follows that for each closed bounded smooth $2(n-r)$-form $\eta$ on $X_K$, by \Cref{Volume Dominance}, there exists a constant $C_\eta$ such that for all $T\geq 0$, 
\begin{equation*}
    |\mathrm{vol}(Z(T,\varphi;K),\eta\wedge \Omega^{r-r(T)})|\leq C_\eta|\mathrm{vol}(Z(T,\varphi;K), \Omega^{n-r(T)})|.
\end{equation*}
Indeed, we first note that 
\begin{equation*}
    |\mathrm{vol}(Z(T,\varphi;K), \Omega^{n-r(T)})|=\sum_{Q(y)=T}\varphi(y)|\mathrm{vol}(Z(y;K),\Omega^{n-r(T)})|
\end{equation*}
by our assumption that $\varphi\geq 0$, and the sign of $\mathrm{vol}(Z(y;K),\Omega^{n-r(T)})$ depends only on the dimension. And for each $T\geq 0$, one applies the lemma to $\alpha=\eta\wedge \Omega^{r-r(T)}$. Note that the constant in \Cref{Volume Dominance} depends only on $\alpha$ and the dimension, and since there are only finitely many of them, the bound $C_\eta$ is independent of $T$.

It follows by \Cref{generating series of volumes} that for any closed smooth bounded $2(n-r)$-form $\eta$, the formal series
\begin{equation*}
    \sum_{T\in Sym_r(\Q)_{\geq0}} \mathrm{vol}(Z(T,\varphi;K), \eta\wedge\Omega^{r-r(T)})q^T
\end{equation*}
converges absolutely. By Zucker's conjecture and \cite[Proposition 5.6]{borel1983laplacian}, there are isomorphisms
    \begin{equation*}
        IH^{2(n-r)}(\X_K;\C)\cong H^{2(n-r)}_{(2)}(X_K;\C)\cong \H^{2(n-r)}_{(2)}(X_K),
    \end{equation*}
where $H^{2(n-r)}_{(2)}(X_K;\C)$ is the $L^2$-cohomology, and $\H^{2(n-r)}_{(2)}(X_K)$ is the space of $L^2$-harmonic forms. By \cite[Lemma 3.10]{bergeron2006proprietes}, any $L^2$-harmonic form $\eta$ on $X_K$ is (smooth) bounded. We view $\eta$ as a linear functional on $IH^{2r}(\X;\C)$ via Zucker's conjecture. By \Cref{geometric volume=degree}, the series
\begin{equation*}
    \sum_{T\in Sym_r(\Q)_{\geq0}} \deg_{X_{K}}(\Icl(Z(T,\varphi;K)\cup \ell^{r-r(T)}\cup [\eta])q^T
\end{equation*}
is equal to the series of volumes above, and therefore converges absolutely for every $L^2$-harmonic $2(n-r)$-form $\eta$. Since $IH^{2r}(\X_K;\C)$ is a finite-dimensional vector space, and $L^2$-harmonic forms exhaust all linear functionals on it by Poincaré duality, we conclude that the generating series of the modular intersection classes converges as well. 
\end{proof}

Altogether, we obtain a cohomological interpretation of Kudla's Siegel--Weil formula, \Cref{generating series of volumes}.

\begin{Theorem}\label{Cohomological interpretation of Siegel--Weil}
    When $r<n+1-\mathrm{Witt}(V)$, we have
    \begin{equation*}
        \deg(\phi_r(\tau,\varphi)\cup\ell^{n-r})=\deg(\Icl(X_K)\cup \ell^n) E_r(\tau, s_0,\varphi).
    \end{equation*}
\end{Theorem}

\begin{proof}
   It follows by \Cref{Modularity of modular intersection class}, \Cref{geometric volume=degree} and \Cref{generating series of volumes} that 
    \begin{equation*}
\begin{aligned}
       \deg(\phi_r(\tau,\varphi)\cup\ell^{n-r}) &= \sum_{T\in Sym_r(\Q)_{\geq0}}\deg(\Icl(Z(T,\varphi;K)\cup \ell^{r-r(T)}\cup \ell^{n-r})q^T\\
       &=\sum_{T\in Sym_r(\Q)_{\geq 0}}\mathrm{vol}(Z(T,\varphi;K),\Omega^{n-r(T)})q^T\\
       &=\deg(\Icl(X_K)\cup\ell^n)E_r(\tau,s_0,\varphi).
\end{aligned}
    \end{equation*}
\end{proof}

We end with a remark on the relation of our construction to Kudla's program.  

\begin{Remark}\label{Kudla remark}
The same argument also applies to the unitary Shimura varieties. Let $X$ be a unitary Shimura variety of signature $(n, 1)$, as considered in \cite{greer2025cohomological}. When $r\leq n/2$, the intersection cohomology $IH^{2r}(\X_K;\C)$ is isomorphic to $W_{2r}H^{2r}(X_K;\C)$, then the  generating series of modular intersection classes $\phi_r(\tau,\varphi)$ is a holomorphic Siegel modular form by \cite[Theorem 3]{kudla1990intersection}. Indeed, Kudla and Millson show that the form $\phi_r(-,\varphi)$ defines an element in 
    \begin{equation*}
M_{n+1}(U(r,r))\otimes W_{2r}H^{2r}(X_K;\C)\cong M_{n+1}(U(r,r))\otimes IH^{2r}(\X_K;\C),
    \end{equation*}
  where $M_{n+1}(U(r,r))$ is the space of Hermitian modular forms of weight $n+1$ of $U(r,r)$. When $r>n/2$ and $X$ is noncompact, the results \cite[Lemma 3.3, Theorem 4.1]{kudla1990intersection} cannot extend to any $L^2$-harmonic form $\eta$, and therefore the theta integral 
  \begin{equation*}
      \int_{X_K}\theta_{r}(\tau,\varphi)\wedge\eta
  \end{equation*}
  may not be holomorphic, and its $T$-th Fourier coefficients may not match with the coefficients of $\phi_r(\tau, \varphi)\cup [\eta]$, the generating series of modular intersection classes. Therefore, $\phi_r(\tau, \varphi)$ may not match with Kudla--Millson's theta function. We note that the coefficients do agree if $T>0$, by the method in \cite[Theorem 2.1, \S 5]{kudla1988tubes}; and applying the functional defined by the volume one does obtain a holomorphic function by the Siegel--Weil formula under Weil's condition. Comparing with the results like \cite[Theorem 1.2, Theorem 1.3]{greer2025cohomological} on the toroidal compactification $X_{\Sigma}^{tor}$, we expect the form $\phi_r(\tau,\varphi)$ to be at least quasi-modular for higher degrees $r>n/2$. In fact, the modular intersection class is a universal class in all smooth compactifications (\Cref{modular intersection class is universal}), and one always has the relation (\Cref{Decomposition is canonical})
  \begin{equation*}
      IH^{2r}(\X;\C)=\textrm{im}(H^{2r}(\hX;\C)\to H^{2r}(X_{\Sigma}^{tor};\C)).
  \end{equation*}
  It will be interesting to compare the strategy here with the (corrections of) cycles considered in \cite{greer2025cohomological}, and the relation between the analytic barrier here and the barrier to \cite[Conjecture 1.7]{greer2025cohomological}.
\end{Remark}

\phantomsection

\addcontentsline{toc}{section}{References}
\bibliography{References}
\end{document}